\documentclass{amsart}

\usepackage{hyperref}
\usepackage{amsmath,arydshln,multirow}
\usepackage{arydshln}
\usepackage{cases}
\usepackage{amsmath}
\usepackage{amsfonts}
\usepackage{bm}
\usepackage{arydshln}
\usepackage{amsfonts,amsmath,amssymb,amscd,bbm,amsthm,mathrsfs,dsfont}
\usepackage{mathrsfs}
\usepackage{pb-diagram}
\usepackage{amssymb}
\usepackage[all,cmtip]{xy}
\usepackage[dvips]{graphicx}
\usepackage{CJK}
\usepackage[dvipsnames]{xcolor}
\usepackage{mathtools}
\usepackage{extarrows}

\usepackage{color}

\newtheorem{theorem}{Theorem}[section]

\newtheorem{proposition}[theorem]{Proposition}
\newtheorem{lemma}[theorem]{Lemma}

\theoremstyle{definition}
\newtheorem{definition}[theorem]{Definition}
\newtheorem{example}[theorem]{Example}
\newtheorem{proposition-definition}[theorem]{Proposition-Definition}
\newtheorem{corollary}[theorem]{Corollary}

\theoremstyle{remark}
\newtheorem{remark}[theorem]{Remark}

\numberwithin{equation}{section}

\begin{document}

\title{The valuation pairing on an upper cluster algebra}



\author{Peigen Cao}

\address{School of Mathematical Sciences, University of Science and Technology of China,
Hefei 230026, Anhui, P. R. China;and
\newline
Graduate School of Mathematics, Nagoya University, Chikusa-ku, Nagoya, 464-8604, Japan; and 
\newline
Universit\'e Paris Cit\'e,
    UFR de Math\'ematiques,
    CNRS,
   Institut de Math\'ematiques de Jussieu--Paris Rive Gauche, IMJ-PRG,
    B\^{a}timent Sophie Germain,
    75205 Paris Cedex 13,
    France
    }

\email{peigencao@126.com}

\author{Bernhard Keller}
\address{ Universit\'e Paris Cit\'e\\
    UFR de Math\'ematiques\\
    CNRS\\
   Institut de Math\'ematiques de Jussieu--Paris Rive Gauche, IMJ-PRG \\
    B\^{a}timent Sophie Germain\\
    75205 Paris Cedex 13\\
    France
}
\email{bernhard.keller@imj-prg.fr}
\urladdr{https://webusers.imj-prg.fr/~bernhard.keller/}

\author{Fan Qin}
\address{
    School of Mathematical Sciences\\
    Shanghai Jiao Tong University\\
    Shanghai 200240\\
    China
}
\email{qin.fan.math@gmail.com}


\date{}

\dedicatory{}


\begin{abstract}
It is known that many (upper) cluster algebras are not unique factorization domains. We exhibit the local factorization properties with respect to any given seed $t$: any non-zero element in a full rank upper cluster algebra can be uniquely written as the product of a cluster monomial in $t$ and another element not divisible by the cluster variables in $t$. Our approach is based on introducing the valuation pairing on an upper cluster algebra: it counts the maximal multiplicity of a cluster variable among the factorizations of any given element. 

We apply the valuation pairing to obtain many results concerning factoriality, $d$-vectors, $F$-polynomials and the combinatorics of cluster Poisson variables. In particular, we obtain that full rank and primitive upper cluster algebras are factorial; an explanation of $d$-vectors using valuation pairing; a cluster monomial in non-initial cluster variables is determined by its $F$-polynomial; the $F$-polynomials of non-initial cluster variables are irreducible; and the cluster Poisson variables parametrize the exchange pairs of the corresponding upper cluster algebra.
\end{abstract}

\maketitle

\addtocontents{toc}{\protect\setcounter{tocdepth}{1}}

\tableofcontents
  

\section{Introduction}

\subsection{Background}
Around the year 2000, Fomin and Zelevinsky introduced  cluster
algebras \cite{FZ} with the aim of developing a combinatorial approach
to the theory of canonical bases in quantum groups and the closely
related theory of total positivity in algebraic groups. Since then,
cluster algebras have been linked to numerous other subjects
and their study has flourished, cf. for example the surveys
\cite{Leclerc10, Fomin10b, Keller12a, Keller19a}.

A cluster algebra $\mathcal A$  is a subalgebra of an ambient field $\mathcal F$ generated by certain combinatorially defined generators (called {\em cluster variables}), which are grouped into overlapping sets (called {\em clusters}) of constant cardinality $n$.   Different clusters are obtained from each other by a sequence of {\em mutations}.  One remarkable feature of cluster algebras
is the Laurent phenomenon \cite{FZ}, that is,  for any given cluster $A_{t_0}=(A_{1;t_0},\ldots,A_{n;t_0})$, each cluster variable $A_{k;t}$  can be written as
 $$A_{k;t}=\frac{P(A_{1;t_0},\ldots,A_{n;t_0})}{A_{1;t_0}^{d_1}\cdots A_{n;t_0}^{d_n}},$$
where $P$ is a polynomial in variables from $A_{t_0}$ such that  $A_{i;t_0}$ does not divide $P$ for any $i$.  The vector $${\bf d}^{t_0}(A_{k;t})=(d_1,\ldots,d_n)^{\rm T}$$ is called  the
 {\em $d$-vector} of $A_{k;t}$ with respect to $A_{t_0}$ and the polynomial $P$ is called the {\em numerator polynomial} of $A_{k;t}$ with respect to $A_{t_0}$.

For each cluster algebra $\mathcal A$,  Fock--Goncharov  defined  \cite{FG0} a pair of varieties: the
{\em cluster $K_2$-variety} and the {\em cluster Poisson variety} (we follow the terminology of the appendix
to \cite{ShenWeng18}). The {\em upper cluster algebra} $\mathcal U$ is defined to be
the algebra of global functions on the cluster $K_2$-variety, and the {\em cluster Poisson algebra} $\mathcal X$ is defined to be  the algebra of the global functions on the cluster Poisson variety. The clusters of $\mathcal U$ correspond to local toric charts and the cluster variables correspond to local
coordinates (which happen to be global functions) of the cluster $K_2$-variety.  The  cluster Poisson variety
admits a canonical atlas of dual toric charts whose coordinates are the {\em cluster
Poisson variables} (but they are often not global) of $\mathcal X$.

\subsection{Main results}

An upper cluster algebra is constructed from its seeds (see Section \ref{sec2} for details). When the upper cluster algebra has geometric type, each seed $t$ corresponds to an extended exchange matrix $\widetilde{B}_t$, see Definition \ref{def:extended_exchange_mat}. The upper cluster algebra is said to be {\em full rank} or {\em primitive} if $\widetilde{B}_t$ is. Such properties are independent of the choice of the seed $t$.

\subsubsection*{Valuation pairings, factoriality and the Ray Fish Theorem}

It is known from \cite{GLS13} that many (upper) cluster algebras are not unique factorization domains.  In order to study the local factorization properties of upper cluster algebras, we introduce the {\em valuation pairing} (see Definition \ref{defvalu}) on any upper cluster algebra $\mathcal U$. To each pair $(A_{k;t},M)$ consisting of a cluster variable $A_{k;t}$ and an element $M$ in $\mathcal U$, it associates the largest integer $s$ (possibly infinity) such that $M/A_{k;t}^s$ still belongs to $\mathcal U$. We write $(A_{k;t}\mid\mid   M)_v=s$. Using the valuation pairing we prove that any full rank upper cluster algebra has the following {\em local unique factorization property}: For each seed $t$ of $\mathcal U$, any non-zero element $M$ can be uniquely factorized as $M=N\cdot L$, where $N$ is a cluster monomial in $t$ and $L$ is an element in $\mathcal U$ not divisible by any cluster variable in $t$.
We give many applications to $d$-vectors, $F$-polynomials, factoriality of upper cluster algebras and combinatorics of cluster Poisson variables.

As an application to factoriality of upper cluster algebras,  we prove that a full rank upper cluster algebra $\mathcal U$ with initial seed $t_0$ is factorial  if and only if each exchange binomial of $t_0$ is irreducible in the corresponding polynomial ring (see Theorem \ref{thmcufd}). In particular, full rank, primitive upper cluster algebras are factorial (see Theorem \ref{thmufd} (i)).  These include principal coefficient upper cluster algebras as a special case. For full rank, primitive upper cluster algebras, we also prove that the numerator polynomials of non-initial cluster variables are irreducible (see Theorem \ref{thmufd} (ii)).

The Starfish Theorem in \cite{BFZ}, cf. also Theorem \ref{proup} of the present paper, plays a very important role in this paper. It states that any full rank upper cluster algebra $\mathcal U$ can be written as the intersection of $n+1$ Laurent polynomial rings. Here $n$ is the rank of $\mathcal U$. Thanks to the results on factoriality of upper cluster algebras, we show that any full rank, primitive upper cluster algebra can be written as the intersection of {\em two} Laurent polynomial rings (see Theorem \ref{thmstrong}). We call this the {\em Ray Fish Theorem}.

\subsubsection*{Application to $d$-vectors} In \cite{CL1}, Li and the first author of the present paper proved that there exists a well-defined function $(-\mid\mid  -)_d$ on the set of cluster variables, which is called the {\em $d$-compatibility degree}. The values of the $d$-compatibility degree $(-\mid\mid  -)_d$ are given by the components of the $d$-vectors. One remarkable property of the $d$-compatibility degree
is that it uniquely determines how the set of cluster variables is grouped into clusters. As an application to $d$-vectors, we show how to express the $d$-compatibility degree and the $d$-vectors using the valuation pairing for full rank upper cluster algebras (see Theorem \ref{thmind}). As an application, in the full rank, primitive case, we prove that if $M$ is a monomial in non-initial cluster variables, then $M$ and its $d$-vector are uniquely determined by the numerator polynomial $P_M$ of $M$ (see Proposition \ref{prodvec}).

\subsubsection*{Application to $F$-polynomials}

Let $B$ be an $n\times n$ skew-symmetrizable integer matrix. The
{\em $F$-polynomial} $F_{k;t}^{B;t_0}$ associated with $(B,t_0; k,t)$ may be constructed by an explicit recursion or using the cluster
algebra $\mathcal A$ with principal coefficients \cite{FZ3} associated with $(B,t_0)$:
Let $A_{t_0}=(A_{1;t_0},\ldots,A_{n;t_0})$ be the initial cluster of $\mathcal A$ and $A_{k;t}$  a cluster variable of $\mathcal A$. In this case, the
cluster variable
$A_{k;t}$ can be written as a Laurent polynomial in $\mathbb Z[Z_1,\ldots,Z_n][A_{1;t_0}^{\pm1},\ldots,A_{n;t_0}^{\pm1}]$. Then the $F$-polynomial of $A_{k;t}$ is the specialization given by
$$F_{k;t}^{B;t_0}(Z_1,\ldots,Z_n):=A_{k;t}\mid _{A_{1;t_0}=\ldots=A_{n;t_0}=1}\in\mathbb Z[Z_1,\ldots,Z_n].$$
When $B$ and $t_0$ are clear, we simply write $F_{k;t}^{B;t_0}$ as $F_{k;t}$.

Let $M=\prod\limits_{i=1}^s M_i$ be a monomial in cluster variables, where each $M_i$ is a cluster variable of $\mathcal A$.
The {\em $F$-polynomial} of $M$ is  the polynomial
$F_M:=\prod\limits_{i=1}^sF_{M_i}$,
where $F_{M_i}$ is the $F$-polynomial of the cluster variable $M_i$.

The $F$-polynomials
are the non tropical ingredients of the canonical expressions
\cite{NZ} for
both, the cluster variables and the cluster Poisson variables.
 They are fundamental in the additive categorification of cluster
algebras (see for example
\cite{CalderoChapoton06,DWZ} and the surveys \cite{BuanMarsh06,
GeissLeclercSchroeer08a,Keller10b, Keller12a, Plamondon18})
and in their link to Donaldson--Thomas theory
(see for example \cite{KontsevichSoibelman08,N,Bridgeland17}). It is known that $F$-polynomials enjoy many nice properties, for example, they have positive coefficients and constant term $1$. We refer the readers to \cite{LeeSchiffler15,DWZ,GHKK} for these results and to \cite{GyodaYurikusa19,Gyoda19,FG,Fei19a,Fei19b,li-pan-2022} for some recent work on $F$-polynomials.

As an application to $F$-polynomials, we prove that  if $M$ is a monomial in non-initial cluster variables, then $M$ is uniquely determined by its $F$-polynomial $F_M$ (see Theorem \ref{thmfpoly}).  We also prove that the $F$-polynomials of non-initial cluster variables are irreducible (see Theorem \ref{thmffirr}).

\subsubsection*{Application to cluster Poisson variables}
  As an application to combinatorics of cluster Poisson variables, we give  several equivalent characterizations of when two cluster Poisson variables are equal (see Theorem \ref{mainthm}). Recall that each mutation $t^\prime=\mu_k(t)$ gives an {\em $\mathscr A$-exchange pair} $(A_{k;t},A_{k;t^\prime})$ of the upper cluster algebra $\mathcal U$ and  an {\em $\mathscr X$-exchange pair} $(X_{k;t},X_{k;t^\prime})$ of the cluster Poisson algebra $\mathcal X$.

 As the first application of Theorem \ref{mainthm}, we prove that
 the cluster Poisson variables of a cluster Poisson algebra $\mathcal X$ parametrize the $\mathscr A$-exchange pairs of the upper cluster algebra $\mathcal U$ of the same type as $\mathcal X$ (see Theorem \ref{thmbij}). This extends the corresponding result by  Sherman-Bennett \cite{S} from the finite type case to full generality.

 As the second application of Theorem \ref{mainthm}, we prove that the $\mathscr X$-seeds of  $\mathcal X$ whose Poisson clusters contain particular cluster Poisson variables form a connected subgraph of the exchange graph of  $\mathcal X$ (see Theorem \ref{thmgraph}). This is analogous to the result on exchange graphs of cluster algebras given in \cite{CL1}, cf. also Theorem \ref{thmacon} of this paper.

\subsection{Contents}

This paper is organized as follows: In Section \ref{sec2}, some basic definitions, notations and known results are introduced. In Section \ref{sec3}, we  introduce the valuation pairing $(-\mid\mid  -)_v$ on any upper cluster algebra and prove the local unique factorization property for full rank  upper cluster algebras (see Theorem \ref{thmunique}). In Sections \ref{sec4}, \ref{sec5}, \ref{sec6}, \ref{sec7}, we give the applications to $d$-vectors, $F$-polynomials,  factoriality of upper cluster algebras and combinatorics of cluster Poisson variables. To be more precise:

In Section \ref{sec41}, we prove that a full rank upper cluster algebra $\mathcal U$ with initial seed $t_0$ is factorial  if and only if each exchange binomial of $t_0$ is irreducible in the corresponding polynomial ring (see Theorem \ref{thmcufd}). In particular, we show in Section \ref{sec42} that full rank, primitive upper cluster algebras are factorial (see Theorem \ref{thmufd} (i)). Moreover, for these upper cluster algebras, we also show that the numerator polynomials of non-initial cluster variables are irreducible (see Theorem \ref{thmufd} (ii)). In Section \ref{sec43}, we give some examples of non-factorial upper cluster algebras. In Section \ref{sec44}, we prove the  Ray Fish Theorem, which states that any full rank, primitive upper cluster algebra $\mathcal U$ can be written as the intersection of {\em two} Laurent polynomial rings (see Theorem \ref{thmstrong}).

In Section \ref{sec51},  we show how to express the $d$-compatibility degree and the $d$-vectors using the valuation pairing for full rank upper cluster algebras (see Theorem \ref{thmind}). In Section \ref{sec52}, we give a local factorization for cluster monomials (see Theorem \ref{thmup}). As an application, in the full rank, primitive case, we prove that if $M$ is a monomial in non-initial cluster variables, then $M$ and its $d$-vector are uniquely determined by the numerator polynomial $P_M$ of $M$ (see Proposition \ref{prodvec}).

 In Section \ref{sec61},  we prove that  if $M$ is a monomial in non-initial cluster variables, then $M$ is uniquely determined by its $F$-polynomial $F_M$ (see Theorem \ref{thmfpoly}).  In Section \ref{sec62}, we  prove that the $F$-polynomials of non-initial cluster variables are irreducible (see Theorem \ref{thmffirr}).

In Section \ref{sec7}, we give several equivalent characterizations of when two cluster Poisson variables are equal (see Theorem \ref{mainthm}).  As the first application, we prove that the cluster Poisson variables of a cluster Poisson algebra $\mathcal X$ parametrize the $\mathscr A$-exchange pairs of the upper cluster algebra $\mathcal U$ of the same type as $\mathcal X$ (see Theorem \ref{thmbij}).  As the second application, we prove that the $\mathscr X$-seeds of  $\mathcal X$ whose Poisson clusters contain particular cluster Poisson variables form a connected subgraph of the exchange graph of  $\mathcal X$ (see Theorem \ref{thmgraph}).

 The following diagram gives  the logical dependence among the proofs of the main theorems in this paper.
$$
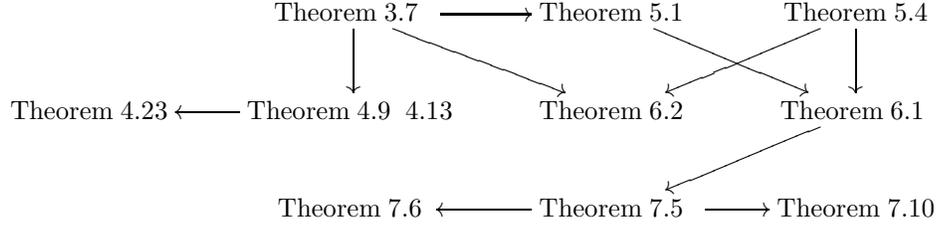
\begin{array}{ccc}
\xymatrix{&{\rm Theorem}\; \ref{thmunique}\;\ar[r]\ar[d]\ar[rd]\;&{\rm Theorem}\; \ref{thmind}\;\ar[rd]\;&{\rm Theorem}\;\ref{thmup}\ar[d]\ar[ld]\\
{\rm Theorem}\;\ref{thmstrong}&{\rm Theorem}\;\ref{thmcufd}\;\;\ref{thmufd}\ar[l]   \;&{\rm Theorem}\;\ref{thmffirr}\;  \;&{\rm Theorem}\; \ref{thmfpoly}\;\ar[ld]\\
&{\rm Theorem}\;\ref{thmbij}   \;&{\rm Theorem}\;\ref{mainthm}\ar[l]\ar[r]\;  \;&{\rm Theorem}\; \ref{thmgraph}}
\end{array}
$$
\subsection{Convention and assumption} \label{rmkexbi}
  Throughout this article,   $\mathbb K$ is assumed to be a factorial domain of characteristic $0$ (e.g., $\mathbb K=\mathbb Z, \mathbb Q, \mathbb R, \mathbb C$)
and all upper cluster algebras are considered as algebras over $\mathbb {KP}$, where $\mathbb P$ is some abelian multiplicative group and $\mathbb{KP}$ the corresponding group ring. The factoriality of $\mathbb K$ is necessary for us, because one of the aims of this paper is to consider the factoriality of upper cluster algebras.

 We always assume that the {\em exchange binomials} of upper cluster algebras in this article are {\em not invertible in $\mathbb {KP}$}. Note that when $\mathbb K=\mathbb Z$, this condition is always satisfied. When $\mathbb K$ is a field and we have a {\em trivial exchange relation} $$A_{k}A_k^\prime=P_k\in(\mathbb{KP})^\times$$ in an upper cluster algebra $\mathcal U$,  we can always freeze the cluster variable $A_k$ and consider the upper cluster algebra $\mathcal U^\dag$ with smaller rank over $\mathbb {KP}^\dag=\mathbb {KP}[A_k^\pm]$. Note that $\mathcal U^\dag$ and $\mathcal U$ are isomorphic as $\mathbb K$-algebras.

\subsection*{Acknowledgements}

The authors are very grateful to Luc Pirio, whose conjectures have inspired a good part of the results of this paper. P. Cao would like to thank Xiaofa Chen, Yu Wang and Yilin Wu for helpful discussions during his stay in Paris. B. Keller is indebted
to Lauren Williams and Melissa Sherman-Bennett for their interest and for stimulating conversations. The authors are grateful to Ana Garcia Elsener and Daniel Smertnig for a helpful email exchange \cite{GS_2022} where, in particular, they defined the function $\lambda$ (cf. also Remark \ref{rmk:atomic} and the Appendix). The authors sincerely thank the anonymous referee for valuable comments including the suggestion of Proposition \ref{prop:compactified_U_factoriality}.

P. Cao thanks the Fondation Sciences Math\'ematiques de Paris (FSMP) for funding his postdoctoral
stay (2019--2020) at Universit\'e de Paris during which the first two authors began work on this paper. He also thanks the European Research Council Grant No. 669655 for funding his postdoctoral stay (2021--2022) at The Hebrew University of Jerusalem. Currently, he is supported by the JSPS Postdoctoral Fellowships for Research in Japan and  the Guangdong Basic and Applied Basic Research
Foundation Grant No. 2021A1515012035. F. Qin thanks the National Natural Science Foundation of China for financial support (Grant No. 12271347).

\section{Preliminaries}\label{sec2}

\subsection{Basics on cluster algebras and cluster Poisson algebras}

An $n\times n$ integer matrix $B$ is said to be {\em  skew-symmetrizable} if there is an integer diagonal matrix $S$ with strictly positive diagonal entries such that $SB$ is skew-symmetric. Such an $S$ is said to be a {\em  skew-symmetrizer} of $B$.

\begin{definition}[Matrix mutation]
Let $\widetilde B=\begin{pmatrix}B_{n\times n}\\C_{m\times n}\end{pmatrix}=(b_{ij})$ be an $(n+m)\times n$ integer matrix with $B$ skew-symmetrizable. The mutation $\mu_k(\widetilde B)$ of $\widetilde B$ at $k\in\{1,\ldots,n\}$ is the new matrix $\mu_k(\widetilde B)=\widetilde B^\prime=
\begin{pmatrix}B_{n\times n}^\prime\\C_{m\times n}^\prime\end{pmatrix}=(b_{ij}^\prime)$ given by
\begin{eqnarray}
b_{ij}^{\prime}=\begin{cases}-b_{ij}~,&\text{if } i=k\text{ or } j=k;\\ b_{ij}+{\rm sgn}(b_{ik}){\rm max}(b_{ik}b_{kj},0)~,&\text{otherwise}.\end{cases}\nonumber
\end{eqnarray}
\end{definition}

It is not hard to check that the submatrix $B^\prime$ of $\widetilde B^\prime$ is still skew-symmetrizable with the same skew-symmetrizer as $B$ and that $\mu_k$ is an involution.

\begin{proposition}
{\rm (}\cite[Lemma 3.2]{BFZ}{\rm ).}
\label{prorank}
Matrix mutations preserve the rank of $\widetilde B$.
\end{proposition}

Recall that $(\mathbb P, \oplus, \cdot)$ is a {\em semifield} if $(\mathbb P,  \cdot)$ is an abelian multiplicative group endowed with a binary operation of auxiliary addition $\oplus$ which is commutative, associative and satisfies that the multiplication  distributes over the auxiliary addition.

The {\em tropical semifield} $\mathbb P={\rm Trop}(Z_1,\ldots,Z_m)$ is the free (multiplicative) abelian group generated by $Z_1,\ldots,Z_m$
with auxiliary addition $\oplus$ defined by
$$\prod\limits_{i}Z_i^{a_i}\oplus\prod\limits_{i}Z_i^{b_i}=\prod\limits_{i}Z_i^{{\rm min}(a_i,b_i)}.$$

Let $\mathbb Q_{\rm sf}(Z_1,\ldots,Z_m)$ be the set of all non-zero rational functions in
$m$ independent variables $Z_1,\ldots,Z_m$, which can be written as subtraction-free rational expressions in
$Z_1,\ldots,Z_m$.  The set $\mathbb Q_{\rm sf}(Z_1,\ldots,Z_m)$  is a semifield
with respect to the usual operations of multiplication and addition. It is called an {\em universal semifield}.

\begin{definition}[$\mathscr X$-seed and cluster Poisson seed]

(i) A {\em (labeled) $\mathscr X$-seed}  over a semifield $\mathbb P$ is a pair $(B,X)$, where
\begin{itemize}
\item $B=(b_{ij})$ is an $n\times n$  skew-symmetrizable integer matrix, called an {\em exchange matrix };

\item $X=(X_1,\ldots,X_n)$ is an $n$-tuple of elements in $\mathbb P$. We call $X$ the {\em  $\mathscr X$-cluster} and  $X_1,\ldots,X_n$ the {\em  $\mathscr X$-variables} of  $(B,X)$.
\end{itemize}

(ii) Let $(B,X)$ be an $\mathscr X$-seed over a semifield $\mathbb P$. If $\mathbb P=\mathbb Q_{\rm sf}(X_1,\ldots,X_n)$, we call $(B,X)$ a {\em cluster Poisson seed}, $X$ the {\em Poisson cluster}, and $X_1,\ldots,X_n$ the {\em cluster Poisson variables} of $(B,X)$.

\end{definition}

Recall that $\mathbb K$ is assumed to be a factorial domain of characteristic $0$. We take the ambient field $\mathcal F$  to be the field of rational functions in $n$ independent variables with coefficients in $\mathbb {KP}$.

\begin{definition}[$\mathscr A$-seed]
  A {\em (labeled)  $\mathscr A$-seed}  over $\mathbb P$ is a triple $(B, X, A)$, where
\begin{itemize}
\item $(B, X)$ forms an  $\mathscr X$-seed over $\mathbb P$;

\item  $A=(A_1,\dots, A_n)$ is an $n$-tuple such that $\{A_1,\dots, A_n\}$ is a free generating set of $\mathcal F$ over $\mathbb{KP}$. We call $A$  the {\em cluster}  and $A_1,\dots,A_n$ the {\em cluster variables}  of  $(B, X, A)$;
\end{itemize}
\end{definition}

\begin{definition}[$\mathscr X$-seed mutation and $\mathscr X$-exchange pair]
Let $(B,X)$ be an   $\mathscr X$-seed over $\mathbb P$.  Define the {\em  $\mathscr X$-seed mutation}  of  $(B,X)$ at  $k\in\{1,\ldots,n\}$ as a new pair $\mu_k(B,X)=(B^\prime, X^{\prime})$, where $B^\prime=\mu_k(B)$ and $X^\prime=(X_1^\prime,\ldots,X_n^\prime)$ is given by
\begin{eqnarray}
 X_i^{\prime}=\begin{cases} X_k^{-1}~,&\text{if } i=k;\\ X_iX_k^{{\rm max}(b_{ki},0)}(1\bigoplus X_k)^{-b_{ki}}~,&\text{if }i\neq k.
\end{cases}\nonumber
\end{eqnarray}
In this case, $(X_k,X_k^\prime)$ is called an {\em  $\mathscr X$-exchange pair}.
\end{definition}

\begin{definition}[$\mathscr A$-seed mutation and $\mathscr A$-exchange pair]
Let  $(B,X, A)$ be an $\mathscr A$-seed over $\mathbb P$. Define the {\em mutation}  of   $(B,X, A)$ at $k\in\{1,\ldots,n\}$ as a new triple $\mu_k(B,X, A)=(B^\prime,X^\prime, A^\prime)$, where $(B^\prime,X^\prime)=\mu_k(B,X)$ and $A^\prime=(A_1^\prime,\ldots,A_n^\prime)$ is given by

$$A_i^{\prime}=\begin{cases}A_i,&\text{if }i\neq k;\\ A_k^{-1}(\frac{X_k}{1\oplus X_k}\prod\limits_{b_{jk}>0}A_j^{b_{jk}}+\frac{1}{1\oplus X_k}\prod\limits_{b_{jk}<0}A_j^{-b_{jk}}),&\text{if }i=k.\end{cases}$$
In this case, $(A_k,A_{k}^\prime)$ is called an {\em $\mathscr A$-exchange pair}. The binomial $$\frac{X_k}{1\oplus X_k}\prod\limits_{b_{jk}>0}A_j^{b_{jk}}+\frac{1}{1\oplus X_k}\prod\limits_{b_{jk}<0}A_j^{-b_{jk}}$$ is called the $k$-th {\em exchange binomial} of  $(B,X, A)$.
\end{definition}
It is not hard to check that each mutation $\mu_k$ maps a seed ($\mathscr X$-seed or cluster Poisson seed or $\mathscr A$-seed) to a new seed of the same type and that $\mu_k$ is an involution.

Let $\mathbb T_n$ be the  $n$-regular tree. Let us label the edges of $\mathbb T_n$ by  $1,\dots,n$ such that the $n$  different edges adjacent to the same vertex of $\mathbb T_n$ receive different labels.
\begin{definition}[Seed pattern]
A {\em seed pattern} $\mathcal S$ over $\mathbb P$ is an assignment of a seed  $\Sigma_t$ ($\mathscr X$-seed or cluster Poisson seed or $\mathscr A$-seed) to every vertex $t$ of the $n$-regular tree $\mathbb T_n$ such that $\Sigma_{t^\prime}=\mu_k(\Sigma_t)$ for any edge $t^{~\underline{\quad k \quad}}~ t^{\prime}$.
\end{definition}
We often fix a vertex $t_0\in\mathbb T_n$ as the rooted vertex of $\mathbb T_n$. For a seed pattern, the seed at the rooted vertex $t_0$ is called the {\em initial seed}.  It is easy to see that a seed pattern is completely determined by its initial seed. 

Now we give some symbols which are used in the sequel.  We always write $B_t=(b_{ij}^t)$, $X_t=(X_{1;t},\ldots,X_{n;t})$ and $A_t=(A_{1;t},\ldots,A_{n;t})$. For simplicity, we will also use $t$ to denote the seed at $t\in\mathbb T_n$.

Two (labeled) seeds are {\em equivalent} if they are the same up to relabeling.

\begin{definition}[Exchange graph]
Let $\mathcal S$ be a seed pattern. The {\em exchange graph} ${\bf EG}(\mathcal S)$ of $\mathcal S$ is a graph whose vertices are in bijection with the seeds (up to equivalence) of $\mathcal S$ and whose edges correspond to the seed mutations.
\end{definition}

\begin{proposition}{\rm (}\cite[Proposition 3.9]{FZ3}{\rm ).}
\label{proxp} Let $\mathcal S$ be an $\mathscr A$-seed pattern. For each $\mathscr A$-seed $t=(B_t,X_t,A_t)$, let $\widehat X_{t}=(\widehat X_{1;t},\ldots,\widehat X_{n;t})$ be the $n$-tuple of elements in $\mathcal F$ given by
 $$\widehat X_{j; t}:=X_{j;t}\prod\limits_{i=1}^nA_{i;t}^{b_{ij}^{t}}.$$ Then $\widehat {\mathcal S}=\{(B_{t},\widehat X_{t})\}_{t\in\mathbb T_n}$ forms an  $\mathscr X$-seed pattern.
\end{proposition}

Now we recall some seed patterns ($\mathscr X$-seed patterns or $\mathscr A$-seed patterns) over particular semifields.
\begin{itemize}
\item A seed pattern  $\mathcal S$  over  $\mathbb P$ is said to be of {\em geometric type}, if $\mathbb P$ is a tropical semifield, say $\mathbb P={\rm Trop}(Z_1,\ldots,Z_m)$.

\item A seed pattern  $\mathcal S$  over  $\mathbb P$ is said to be with {\em principal coefficients} at $t_0$, if $\mathbb P={\rm Trop}(Z_1,\ldots,Z_n)$ and $X_{i;t_0}=Z_i$ for $i=1,\ldots,n$.

    \item A seed pattern  $\mathcal S$ over  $\mathbb P$ is said to be with {\em universal coefficient semifield}, if $\mathbb P=\mathbb {Q}_{\rm {sf}}(X_{1;t_0},\ldots,X_{n;t_0})$ for some seed $t_0$ of $\mathcal S$.
\end{itemize}

\begin{definition}[Coefficient matrices and extended exchange matrices] \label{def:extended_exchange_mat}
Let $\mathcal S$ be a seed pattern of geometric type with coefficient semifield $\mathbb P={\rm Trop}(Z_1,\ldots,Z_m)$. We know that each  $\mathscr X$-variable $X_{k;t}$ has the form $$X_{k;t}=Z_1^{c_{1k}^t}\cdots Z_m^{c_{mk}^t}.$$
The matrix $C_t:=(c_{ij}^t)_{m\times n}$ is called the {\em coefficient matrix} at $t$ and the $(n+m)\times n$ matrix $\widetilde B_t=\begin{pmatrix}B_t\\C_t\end{pmatrix}$ is called the {\em extended exchange matrix} at $t$.
\end{definition}

\begin{proposition}
{\rm (}\cite[Proposition 5.8]{FZ}{\rm ).}
Let $\mathcal S$ be a seed pattern  of geometric type. Then for any edge $t^{~\underline{\quad k \quad}}~ t^{\prime}$ in $\mathbb T_n$, we have $\widetilde B_{t^\prime}=\mu_k(\widetilde B_t)$, where $\widetilde B_t$ and $\widetilde B_{t^\prime}$ are the extended exchange matrices at $t$ and $t^\prime$.
\end{proposition}

\begin{definition}[Cluster Poisson algebra] \label{defcpa}
Let $\mathcal S_{\rm uc}$ be an $\mathscr X$-seed pattern with universal coefficient semifield.  The {\em cluster Poisson algebra} $\mathcal X=\mathcal X(\mathcal S_{\rm uc})$ associated with $\mathcal S_{\rm uc}$ is the intersection
$$\mathcal X=\bigcap\limits_{t\in\mathbb T_n} {\mathcal L}(t),$$
where ${\mathcal L}(t)=\mathbb {K}[X_{1;t}^{\pm1},\ldots,X_{n;t}^{\pm1}]$.
\end{definition}

\begin{definition}[Cluster algebra and upper cluster algebra]
Let $\mathcal S$ be an $\mathscr A$-seed pattern over a semifield $\mathbb P$.

(i) The {\em cluster algebra} $\mathcal A=\mathcal A(\mathcal S)$ associated with $\mathcal S$ is the $\mathbb {KP}$-subalgebra of $\mathcal F$ generated by the cluster variables of $\mathcal S$, namely, $\mathcal A=\mathbb{KP}[A_{1;t},\ldots,A_{n;t}\mid t\in\mathbb T_n]$.

(ii) The {\em upper cluster algebra} $\mathcal U=\mathcal U(\mathcal S)$ associated with $\mathcal S$ is the intersection
$$\mathcal U=\bigcap\limits_{t\in\mathbb T_n} {\mathcal L}(t),$$
where ${\mathcal L}(t)=\mathbb {KP}[A_{1;t}^{\pm1},\ldots,A_{n;t}^{\pm1}]$.
\end{definition}

Since cluster Poisson algebras and (upper) cluster algebras are defined from seed patterns, we can talk about the exchange graphs of these algebras. We can also talk about the (upper) cluster algebras of geometric type, with principal coefficients and with universal coefficient semifield. 

\begin{theorem}[Laurent phenomenon and  positivity] Let $\mathcal A$ be a cluster algebra with coefficient semifield $\mathbb P$ and initial seed $t_0$. The following statements hold.
\begin{itemize}
\item[(i)] {\rm (}\cite{FZ}{\rm ).} Each  cluster variable $A_{k;t}$ of $\mathcal A$ can be written as a  Laurent polynomial in $\mathbb {ZP}[A_{1;t_0}^{\pm1},\ldots,A_{n;t_0}^{\pm1}]$.

    \item[(ii)] {\rm (}\cite{GHKK}{\rm ).} The coefficients of the above Laurent polynomial are in $\mathbb{ NP}$.

    \item[(iii)] {\rm (}\cite{FZ1}{\rm ).} If  $\mathbb P={\rm Trop}(Z_1,\ldots,Z_m)$, then each cluster  variable $A_{k;t}$ of $\mathcal A$ can be written as a  Laurent polynomial in
$\mathbb Z[Z_1,\ldots,Z_m][A_{1;t_0}^{\pm1},\ldots,A_{n;t_0}^{\pm1}]$.
\end{itemize}
\end{theorem}

A geometric upper cluster algebra with initial seed $t_0$ is called a {\em full rank upper cluster algebra} if its initial extended exchange matrix $\widetilde B_{t_0}$ has full rank.

\begin{remark} For a full rank upper cluster algebra  $\mathcal U$, we know that every extended exchange matrix $\widetilde B_t$ of $\mathcal U$ has full rank, by  Proposition \ref{prorank}.
\end{remark}

\begin{theorem}{\rm (}Starfish Theorem, \cite[Corollary 1.9]{BFZ}{\rm ).} \label{proup}
Let $\mathcal U$  be a full rank upper cluster algebra and $t_0$ a seed of $\mathcal U$. Then we have $$\mathcal U=\bigcap\limits_{i=0}^n{\mathcal L}(t_i),$$ where $t_j=\mu_j(t_0)$ for $j=1,\ldots,n$ and ${\mathcal L}(t_i)=\mathbb{KP}[A_{1;t_i}^{\pm1},\ldots,A_{n;t_i}^{\pm1}]$ for $i=0,1,\ldots,n$.
\end{theorem}

\begin{theorem}\label{thmacon}
{\rm (}\cite[Theorem 10]{CL1}{\rm ).}
Let $\mathcal A$ be a cluster algebra with coefficient semifield $\mathbb P$. Then the seeds of  $\mathcal A$ whose clusters contain particular cluster variables form a connected subgraph of the exchange graph  of $\mathcal A$.
\end{theorem}

 The following result is a direct corollary.
\begin{corollary}{\rm (}\cite[Corollary 3]{CL1}{\rm ).}
\label{corn-1} 
Let $\mathcal A$ be a cluster algebra with coefficient semifield $\mathbb P$ and $t_1,t_2$ two seeds of $\mathcal A$. If $t_1$ and $t_2$ have at least $n-1$ common cluster variables, then in the exchange graph of $\mathcal A$ either $t_1$ and $t_2$ represent the same vertex or there is an edge between $t_1$ and $t_2$.
\end{corollary}

Recall that a cluster algebra is {\em acyclic} if it has an acyclic seed: a seed with an exchange matrix $B=(b_{ij})$, such
that there is no sequence of indices $i_1,\ldots,i_{\ell+1}\in\{1,\ldots,n\}$ with  $i_{\ell+1}=i_1$ and $b_{i_j,i_{j+1}}>0$  for $j=1,\ldots,\ell$.

\begin{proposition}{\rm (}\cite{BFZ}, \cite[Theorem 2]{M}{\rm ).}
\label{proacyclic}
Let $\mathcal A$ be an acyclic cluster algebra and $\mathcal U$ the corresponding upper cluster algebra. Then $\mathcal A=\mathcal U$.
\end{proposition}

\subsection{\texorpdfstring{$d$}{Lg}-vectors, \texorpdfstring{$f$}{Lg}-vectors, \texorpdfstring{$g$}{Lg}-vectors, \texorpdfstring{$c$}{Lg}-vectors and \texorpdfstring{$F$}{Lg}-polynomials}
Let $M$ be a non-zero element in ${\mathcal L}(t_0)=\mathbb{KP}[A_{1;t_0}^{\pm1},\ldots,A_{n;t_0}^{\pm1}]$. Then $M$ has the form
\begin{eqnarray}
M=\frac{P_M(A_{1;t_0},\ldots,A_{n;t_0})}{A_{1;t_0}^{d_1}\cdots A_{n;t_0}^{d_n}},\nonumber
\end{eqnarray}
where $P_M\in\mathbb {KP}[A_{1;t_0},\ldots,A_{n;t_0}]$ with  $A_{j;t_0}\nmid P_M$ for $j=1,\ldots,n$.
The vector $${\bf d}^{t_0}(M) = (d_1,\ldots, d_n)^{\rm T}$$ is called the  {\em $d$-vector} of  $M$ with respect to $A_{t_0}=(A_{1;t_0},\ldots,A_{n;t_0})$ and the polynomial $P_M$ is called the {\em numerator polynomial} of $M$ with respect to $A_{t_0}$.

\begin{remark}\label{rmkdadd}
Notice that for any $0\neq M,N\in\mathcal L(t_0)$, we have $${\bf d}^{t_0}(M\cdot N)={\bf d}^{t_0}(M)+{\bf d}^{t_0}(N).$$
\end{remark}

\begin{definition}[$F$-polynomial]\label{deffp}
Let $B$ be an $n\times n$ skew-symmetrizable integer matrix and $\mathcal A$  a principal coefficient cluster algebra at $t_0$ with $B_{t_0}=B$.
 Let $A_{k;t}$ be a cluster variable of $\mathcal A$, which can be written as  a Laurent polynomial in $$\mathbb Z[Z_1,\ldots,Z_n][A_{1;t_0}^{\pm1},\ldots,A_{n;t_0}^{\pm1}],$$ by the Laurent phenomenon. The polynomial
$$F_{k;t}^{B;t_0}(Z_1,\ldots,Z_n):=A_{k;t}\mid _{A_{1;t_0}=\ldots=A_{n;t_0}=1}\in\mathbb Z[Z_1,\ldots,Z_n],$$
which only depends on $(B,t_0;k,t)$, is called an {\em $F$-polynomial} of $B$.
{\em Sometimes we view $F_{k;t}^{B;t_0}$ as a polynomial in new variables $Y_1,\ldots,Y_n$ due to the convention in  \cite{FZ3}. }
\end{definition}
\begin{proposition}{\rm (}\cite[Proposition 5.2]{FZ3}{\rm ).} \label{profy}The $F$-polynomial
$F_{k;t}^{B;t_0}(Y_1,\ldots,Y_n)$ is not divisible by any $Y_j$.
\end{proposition}

Let $B$ be an $n\times n$ skew-symmetrizable matrix and $t_0$  a vertex of $\mathbb T_n$. Now we define two families of integer matrices $\{D_t^{B;t_0}\}_{t\in\mathbb T_n}$ and $\{F_t^{B;t_0}\}_{t\in\mathbb T_n}$.
\begin{definition}[$D$-matrices and $F$-matrices] \label{defdf}
Let $B$ be an $n\times n$ skew-symmetrizable matrix and $t_0$  a vertex of $\mathbb T_n$.

(i) {\rm (}\cite[(7.6), (7.7)]{FZ3}{\rm ).} The matrix $D_t^{B;t_0}=(d_{ij;t}^{B;t_0})$ is uniquely determined by the initial condition $B_{t_0}=B$ and $D_{t_0}^{B;t_0}=-I_n$, together with the following recurrence relations:
\begin{equation}
d_{ij;t^\prime}^{B;t_0}=\begin{cases}d_{ij;t}^{B;t_0},  & \text{if } j\neq k;\\
-d_{ik;t}^{B;t_0}+{\rm max}\{\sum\limits_{b_{lk}^t>0}d_{il;t}^{B;t_0} b_{lk}^t, \sum\limits_{b_{lk}^t<0} -d_{il;t}^{B;t_0}b_{lk}^t\},  &\text{if } j=k,\end{cases}\nonumber
 \end{equation}
 for any edge $t^{~\underline{\quad k \quad}}~ t^{\prime}$ in $\mathbb T_n$. The matrices $\{D_t^{B;t_0}\}_{t\in\mathbb T_n}$ are called the {\em $D$-matrices} of $B$.

(ii) {\rm (}\cite[Definition 2.6]{FG1}{\rm ).} Let $F_{k;t}^{B;t_0}(Z_1,\ldots,Z_n)$ be the $F$-polynomial given in Definition \ref{deffp}, and $f_{ik;t}^{B;t_0}$ be the maximal exponent of $Z_i$ appearing in $F_{k;t}^{B;t_0}$. The vector $${\bf f}_{k;t}^{B;t_0}=(f_{1k;t}^{B;t_0},\ldots,f_{nk;t}^{B;t_0})^{\rm T}$$ is called an {\em $f$-vector} of $B$ and the matrix $F_t^{B;t_0}=(f_{ij;t}^{B;t_0})_{n\times n}$ is called an {\em $F$-matrix} of $B$.
\end{definition}

\begin{remark}
Let $\mathcal A$ be a cluster algebra with initial exchange matrix $B$ at $t_0$. Then the $k$-th column vector ${\bf d}_{k;t}^{B;t_0}$ of $D_{t}^{B;t_0}$ is exactly the $d$-vector ${\bf d}^{t_0}(A_{k;t})$ of  $A_{k;t}$ with respect to $A_{t_0}$, by \cite[(7.7)]{FZ3}.
\end{remark}

\begin{proposition}\label{prodvectors}
Let $\mathcal U$ be an  upper cluster algebra and $M$ a non-zero element in $\mathcal U$. Let $t_0, t_1$ be two seeds of $\mathcal U$ with $t_1=\mu_k(t_0)$. Let ${\bf d}=(d_1,\ldots,d_n)^{\rm T}$ and ${\bf d}^\prime=(d_1^\prime,\ldots,d_n^\prime)^{\rm T}$ be the $d$-vectors of $M$ with respect to $t_0$ and $t_1$ respectively.  Then $d_i=d_i^\prime$ for any $i\neq k$.
\end{proposition}
\begin{proof}
The proof is the same as that of \cite[Proposition 2.5]{RS}.
\end{proof}

\begin{theorem} [Canonical expressions]\label{thmsf}
Let $B$ be an $n\times n$ skew-symmetrizable matrix and  $$\{F_{1;t}^{B;t_0},\ldots,F_{n;t}^{B;t_0}\}_{t\in\mathbb T_n}$$  the $F$-polynomials given in Definition \ref{deffp}.
\begin{itemize}
 \item[(i)] {\rm (}\cite[Corollary 6.3]{FZ3}{\rm ).} Let $\mathcal A$  be a cluster algebra  with principal coefficients at $t_0$ and with $B_{t_0}=B$. Then for any cluster variable $A_{k;t}$ of $\mathcal A$, there exists a unique vector ${\bf g}_{k;t}^{B;t_0}=(g_{1k;t}^{B;t_0},\ldots,g_{nk;t}^{B;t_0})^{\rm T}\in \mathbb Z^n$ such that
\begin{eqnarray}\label{eqnsf}
A_{k;t}=A_{t_0}^{{\bf g}_{k;t}^{B;t_0}}\cdot F_{k;t}^{B;t_0}(\widehat X_{1;t_0},\ldots,\widehat X_{n;t_0}),\nonumber
\end{eqnarray}
where $A_{t_0}^{{\bf g}_{k;t}^{B;t_0}}=\prod\limits_{i=1}^nA_{i;t}^{g_{ik;t}^{B;t_0}}$, $\widehat X_{j;t_0}=X_{j;t_0}\prod\limits_{i=1}^nA_{i;t_0}^{b_{ij}^{t_0}}=Z_j\prod\limits_{i=1}^nA_{i;t_0}^{b_{ij}^{t_0}}$.
\item[(ii)] {\rm (}\cite[Proposition 1.1]{NZ}{\rm ).} Let $\mathcal X$  be the cluster Poisson algebra with  $B_{t_0}=B$. Then for any cluster Poisson variable $X_{k;t}$ of $\mathcal X$, there exists a unique vector ${\bf c}_{k;t}^{B;t_0}=(c_{1k;t}^{B;t_0},\ldots,c_{nk;t}^{B;t_0})^{\rm T}\in \mathbb Z^n$ such that
$$X_{k;t}=X_{t_0}^{{\bf c}_{k;t}^{B;t_0}}\cdot \prod\limits_{i=1}^n(F_{i;t}^{B;t_0}(X_{1;t_0},\ldots,X_{n;t_0}))^{b_{ik}^t}.$$
\end{itemize}
\end{theorem}

\begin{definition}[$g$-vectors and $c$-vectors]
 Keep the notations of Theorem \ref{thmsf}.

 (i) The vector ${\bf g}_{k;t}^{B;t_0}=(g_{1k;t}^{B;t_0},\ldots,g_{nk;t}^{B;t_0})^{\rm T}\in \mathbb Z^n,$
 which only depends on $(B,t_0;k,t)$, is called a {\em $g$-vector} of $B$. The matrix $G_t^{B;t_0}=({\bf g}_{ij;t}^{B;t_0})_{n\times n}$ is called a {\em $G$-matrix} of $B$ at $t$.

 (ii) The vector ${\bf c}_{k;t}^{B;t_0}=(c_{1k;t}^{B;t_0},\ldots,c_{nk;t}^{B;t_0})^{\rm T}\in \mathbb Z^n,$
 which only depends on $(B,t_0;k,t)$, is called a {\em $c$-vector} of $B$. The matrix $C_t^{B;t_0}=({\bf c}_{ij;t}^{B;t_0})_{n\times n}$ is called a {\em $C$-matrix} of $B$ at $t$.
\end{definition}

\begin{theorem}{\rm (}\cite[(3.11)]{NZ}, \cite[Corollary 4.11]{CL}{\rm ).}\label{thmcg} Let $B$ be an $n\times n$ skew-symmetrizable matrix and $S$ a skew-symmetrizer of $B$. Then for any vertices $t_0, t\in\mathbb T_n$, we have
$$SC_t^{B;t_0}S^{-1}(G_t^{B;t_0})^{\rm T}=I_n.$$
\end{theorem}

\subsection{Compatibility degrees on the set of cluster variables}
In this subsection, we recall that the $d$-compatibility degree given in \cite{CL1} and $f$-compatibility degree given in \cite{FG}.

\begin{definition}[Compatibility degrees]
Let $\mathcal A$ be a cluster algebra and $\mathbb A$  the set of cluster variables of $\mathcal A$.

(i) {\rm (}\cite[Definition 8]{CL1}{\rm ).} The $d$-compatibility degree $(-\mid\mid  -)_d: \mathbb A\times \mathbb A\rightarrow\mathbb Z_{\geq-1}$ is defined by
$$(A_{i;t_0}\mid\mid  A_{j;t})_d:=d_{ij;t}^{B_{t_0};t_0},$$
where  $d_{ij;t}^{B_{t_0};t_0}$ is the $(i,j)$-entry of the $D$-matrix $D_t^{B_{t_0};t_0}$.

(ii) {\rm (}\cite[Definition 4.9]{FG}{\rm ).} The $f$-compatibility degree $(-\mid\mid  -)_f: \mathbb A\times \mathbb A\rightarrow\mathbb N$ is defined by
$$(A_{i;t_0}\mid\mid  A_{j;t})_f:=f_{ij;t}^{B_{t_0};t_0},$$
where  $f_{ij;t}^{B_{t_0};t_0}$ is the $(i,j)$-entry of the $F$-matrix $F_t^{B_{t_0};t_0}$.
\end{definition}
Now we summarize some properties of the $d$-compatibility degree $(-\mid\mid  -)_d$ and the $f$-compatibility degree $(-\mid\mid  -)_f$. 
\begin{proposition}
\label{rmkdf}
{\rm (}\cite[Theorem 11, Remark 2]{CL1}, \cite[Theorems 3.3, 4.18]{FG}{\rm ).} The $d$-compatibility degree $(-\mid\mid  -)_d$ and the $f$-compatibility degree $(-\mid\mid  -)_f$ have the following properties.
\begin{itemize}
\item[(1)]  They are well-defined.

\item[(2)] The following statements are equivalent.
\begin{itemize}
\item[(i)] There exists a cluster $A_{t^\prime}$ containing both $A_{i;t_0}$ and $A_{j;t}$;

\item[(ii)] $(A_{i;t_0}\mid\mid  A_{j;t})_d\leq 0$;

\item[(iii)] $(A_{j;t}\mid\mid  A_{i;t_0})_d\leq 0$;

\item[(iv)]  $(A_{i;t_0}\mid\mid  A_{j;t})_f= 0$;

\item[(v)] $(A_{j;t}\mid\mid  A_{i;t_0})_f=0$.
\end{itemize}

\item[(3)] The following statements are equivalent.

\begin{itemize}
\item[(i)] There exists no cluster $A_{t^\prime}$ containing both $A_{i;t_0}$ and $A_{j;t}$;

\item[(ii)] $(A_{i;t_0}\mid\mid  A_{j;t})_d>0$;

\item[(iii)] $(A_{j;t}\mid\mid  A_{i;t_0})_d>0$;

\item[(iv)]  $(A_{i;t_0}\mid\mid  A_{j;t})_f>0$;

\item[(v)]$(A_{j;t}\mid\mid  A_{i;t_0})_f>0$.
\end{itemize}
\item[(4)] The following statements are equivalent.
\begin{itemize}
\item[(i)] $A_{j;t}=A_{i;t_0}$;

\item[(ii)] $(A_{i;t_0}\mid\mid  A_{j;t})_d<0$;

\item[(iii)]$(A_{j;t}\mid\mid  A_{i;t_0})_d<0$;
\item[(iv)] $(A_{i;t_0}\mid\mid  A_{j;t})_d=-1$;

\item[(v)] $(A_{j;t}\mid\mid  A_{i;t_0})_d=-1$.
\end{itemize}

\item[(5)] The following statements are equivalent.
\begin{itemize}
\item[(i)] $A_{i;t_0}\neq A_{j;t}$ and there exists a cluster $A_{t^\prime}$ containing both $A_{i;t_0}$ and $A_{j;t}$;

\item[(ii)] $(A_{i;t_0}\mid\mid  A_{j;t})_d=0$;
\item[(iii)] $(A_{j;t}\mid\mid  A_{i;t_0})_d=0$;
\end{itemize}
\end{itemize}
\end{proposition}

\begin{remark}\label{rmkdadd2}
(i) Roughly speaking, the integer $(A_{i;t_0}\mid\mid  A_{j;t})_d$ is defined to be the $i$-th component of the $d$-vector of $A_{j;t}$ with respect to $A_{t_0}$. By Proposition \ref{prodvectors} and Theorem \ref{thmacon}, the $d$-compatibility degree $(-\mid\mid  -)_d$ is actually a well-defined function on $\mathbb A\times (\mathcal U\backslash\{0\})$, where $\mathcal U$ is the corresponding upper cluster algebra.

(ii) By Remark \ref{rmkdadd}, for any $0\neq M,N\in\mathcal U$, we have $$(A_{i;t_0}\mid\mid  M\cdot N)_d=(A_{i;t_0}\mid\mid  M)_d+(A_{i;t_0}\mid\mid  N)_d.$$
\end{remark}

Recall that a {\em cluster monomial} in a seed $t$  is a monomial in the cluster variables from $t$.
\begin{corollary}\label{cordmonomial}
Let $\mathcal A$ be a cluster algebra with coefficient semifield $\mathbb P$ and $A_{k;t_0}$ a cluster variable of $\mathcal A$. Let $t$ be a seed of $\mathcal A$ and $M=A_{1;t}^{c_1}\cdots A_{n;t}^{c_n}$ a cluster monomial in $t$. Denote $I=\{i\mid c_i>0\}$. The following statements hold.
\begin{itemize}
\item[(i)]If $(A_{k;t_0}\mid\mid  M)_d<0$, then there exists $j_0\in I$ such that $A_{k;t_0}=A_{j_0;t}$. Namely, $A_{k;t_0}$ appears in $M$.

\item[(ii)] If $(A_{k;t_0}\mid\mid  M)_d\leq0$, then $(A_{k;t_0}\mid\mid  A_{i;t})_d\leq0$ for any $i\in I$.
\end{itemize}
\end{corollary}
\begin{proof}
(i) By $c_1,\ldots,c_n\geq0$ and $(A_{k;t_0}\mid\mid  M)_d<0$, we have
$$(A_{k;t_0}\mid\mid  M)_d=\sum_{i=1}^nc_i(A_{k;t_0}\mid\mid  A_{i;t})_d=\sum_{i\in I}c_i(A_{k;t_0}\mid\mid  A_{i;t})_d<0.$$
So there must exist some $j_0\in I$ such that $(A_{k;t_0}\mid\mid  A_{j_0;t})_d<0$. Then by Proposition \ref{rmkdf} (4), we get $A_{k;t_0}=A_{j_0;t}$.

(ii) Assume by contradiction there exists some $i_0\in I$ such that
$$(A_{k;t_0}\mid\mid  A_{i_0;t})_d>0.$$ Then by $(A_{k;t_0}\mid\mid  M)_d=\sum\limits_{i\in I}c_i(A_{k;t_0}\mid\mid  A_{i;t})_d\leq0$, there must exist some $i_1\in I$ such that $(A_{k;t_0}\mid\mid  A_{i_1;t})_d<0$. So we get $A_{k;t_0}=A_{i_1;t}$, by Proposition \ref{rmkdf} (4). Thus $A_{i_0;t}$ and $A_{k;t_0}=A_{i_1;t}$ are in the same cluster. Then by  Proposition \ref{rmkdf} (2), we get $(A_{k;t_0}\mid\mid  A_{i_0;t})_d\leq 0$. This contradicts  $(A_{k;t_0}\mid\mid  A_{i_0;t})_d>0$. So $(A_{k;t_0}\mid\mid  A_{i;t})_d\leq0$ for any $i\in I$.
\end{proof}

Let $\mathcal A$ be a cluster algebra and $\mathbb A$  the set of cluster variables of $\mathcal A$.
We say that two cluster variables $A_{i;t_0}$ and $A_{j;t}$ of $\mathcal A$ are {\em compatible}  if there exists a cluster $A_{t^\prime}$ of $\mathcal A$ containing both $A_{i;t_0}$ and $A_{j;t}$, which is equivalent to $(A_{i;t_0}\mid\mid  A_{j;t})_d\leq 0$, by Proposition \ref{rmkdf} (2). A subset $U$ of $\mathbb A$  is called a  {\em compatible set} if any two cluster variables in $U$ are compatible.
\begin{theorem}{\rm (}\cite[Theorem 13]{CL1}{\rm ).}  \label{thmcomp}
Let $\mathcal A$ be a cluster algebra and $\mathbb A$  the set of cluster variables of  $\mathcal A$. Then
\begin{itemize}
\item[(i)] a subset $U$ of $\mathbb A$ is a compatible set if and only if it is a subset of some  cluster of $\mathcal A$;
\item[(ii)] a subset $U$ of $\mathbb A$ is a maximal compatible set if and only it is a cluster of $\mathcal A$.
\end{itemize}
\end{theorem}

\subsection{Unique factorization domains}
 Recall that an integral domain is a non-zero commutative
ring in which the product of any two nonzero elements is nonzero. From now on, we always assume that $R$ is an integral domain and we denote by $R^\times$  the set of invertible elements in $R$.

Two non-zero elements $r_1,r_2\in R$
are {\em associate} if there exists some invertible element $s\in R^\times$ such that $r_2=sr_1$. A non-zero, non-invertible element $r$ in  $R$ is {\em irreducible}, if
any factorization $r=r_1r_2$ with $r_1,r_2\in R$ implies that either $r_1$ or $r_2$ belongs to $R^\times$. A non-zero, non-invertible element $r$ in $R$ is {\em prime},  if
whenever $r\mid r_1r_2$ for some  $r_1,r_2\in R$, then $r\mid r_1$ or $r\mid r_2$. Every prime element is
irreducible, but the converse is not true in general.

\begin{definition}
 An integral domain $R$
is  {\em factorial}  if the following hold:
\begin{itemize}
\item[(i)] Every non-zero, non-invertible element $r\in R$ can be written as a product
$r=a_1\cdots a_s$ of irreducible elements $a_i\in R$.
\item[(ii)] If $a_1\cdots a_s=b_1\cdots b_t$ with $a_i,b_j\in R$ irreducible for all $i$ and $j$, then
$s=t$ and there is a bijection $\sigma:\{1,\ldots,s\}\rightarrow \{1,\ldots,s\}$ such that $a_i$
and $b_{\sigma(i)}$ are associate for all $1\leq i\leq s$.
\end{itemize}

Factorial domains are also known as {\em unique factorization domains}.
\end{definition}
It is easy to see that in a factorial domain, all irreducible elements are prime.

\begin{theorem} \label{thmgls}
Let $\mathcal A$ be a geometric cluster algebra  with coefficient semifield $\mathbb P={\rm Trop}(Z_1,\ldots,Z_m)$ and
$\mathcal U$ the corresponding upper cluster algebra.
Then the following statements hold.
\begin{itemize}
\item[(i)] 
\begin{itemize}
\item[(a)] $\mathcal A^\times=\{\lambda Z_1^{c_1}\cdots Z_m^{c_m}\mid \lambda\in\mathbb K^\times,\; c_1,\ldots,c_m\in \mathbb Z\}$.
\item[(b)] Any cluster variable  is irreducible in $\mathcal A$.
\end{itemize}
\item[(ii)]
\begin{itemize}
\item[(a)] $\mathcal U^\times=\{\lambda Z_1^{c_1}\cdots Z_m^{c_m}\mid \lambda\in\mathbb K^\times,\; c_1,\ldots,c_m\in \mathbb Z\}$.
\item[(b)] Any cluster variable is irreducible  in $\mathcal U$.
\end{itemize}
\end{itemize}
\end{theorem}

\begin{proof}
(i) This is a result in \cite[Theorem 1.3]{GLS13}.

(ii) The proof is same as that of (i) given in  \cite{GLS13}.
\end{proof}

\section{The valuation pairing on an upper cluster algebra}
\label{sec3}
In this section, we  introduce the valuation pairing $(-\mid\mid  -)_v$ on an upper cluster algebra and prove the local unique factorization property for full rank  upper cluster algebras.

\begin{definition}[Valuation pairing] \label{defvalu} Let $\mathcal U$ be an upper cluster algebra and $\mathbb A$ the set of cluster variables of $\mathcal U$. The pairing
\begin{eqnarray}
(-\mid\mid  -)_v:\mathbb A\times\mathcal U&\rightarrow&\mathbb N\cup\{\infty\}\nonumber\\
(A_{k;t},M)&\mapsto&{\rm max}\{s\in\mathbb N\mid M/A_{k;t}^s\in\mathcal U\}\nonumber
\end{eqnarray}
is called the {\em valuation pairing} on $\mathcal U$.
\end{definition}
It is easy to see that $(A_{k;t}\mid\mid  M)_v=0$ if and only if $M/A_{k;t}\notin \mathcal U$.

\begin{example}
Let $\mathcal U$ be an upper cluster algebra with a seed $t$ and  $M$ a cluster variable in $t$. Then $$(A_{k;t}\mid\mid  M)_v=\begin{cases}1,&\text{if}\;M=A_{k;t};\\0,&\text{if}\;M=A_{j;t}\;\text{and}\;j\neq k.\end{cases}$$
\end{example}

Now we summarize some useful and easy facts on the valuation pairing in the following proposition.

\begin{proposition}\label{proind}
  Let $\mathcal U$ be an upper cluster algebra and $t_0$ a seed of $\mathcal U$.  Let $M$ and $L$ be two elements in $\mathcal U$. The following statements hold.
  \begin{itemize}
  \item[(i)] If $M/A_{k;t_0}^m\notin \mathcal U$ for some $m\in\mathbb Z$, then we also have $M/A_{k;t_0}^{m+1}\notin \mathcal U$.
  \item[(ii)] If $M\neq 0$, then there exists $m>0$ such that $$M/A_{k;t_0}^m\notin \mathcal U.$$ In particular,
  $(A_{k;t_0}\mid\mid  M)_v<\infty$.
  \item[(iii)] $(A_{k;t_0}\mid\mid  M)_v=\infty$ if and only if $M=0$.
  \item[(iv)] $(A_{k;t_0}\mid\mid  M+L)_v\geq{\rm min}\{(A_{k;t_0}\mid\mid  M)_v,(A_{k;t_0}\mid\mid  L)_v\}$.

\item[(v)] For any $s\geq 0$, we have $(A_{k;t_0}\mid\mid   A_{k;t_0}^s\cdot M)_v=s+(A_{k;t_0}\mid\mid  M)_v$.
  \end{itemize}
\end{proposition}
\begin{proof}
(i) Assume by contradiction that $M/A_{k;t_0}^{m+1}\in \mathcal U$, then $$M/A_{k;t_0}^m=A_{k;t_0}\cdot (M/A_{k;t_0}^{m+1})\in\mathcal U.$$ This is a contradiction. So $M/A_{k;t_0}^{m+1}\notin \mathcal U$.

(ii) Without loss of generality, we can assume that $k=n$.
As an element in $${\mathcal L}(t_0)=\mathbb {KP}[A_{1;t_0}^{\pm 1},\ldots,A_{n;t_0}^{\pm 1}]=\mathbb {KP}[A_{1;t_0}^{\pm 1},\ldots, A_{n-1;t_0}^{\pm 1}][A_{n;t_0}^{\pm 1}],$$
 $M$ can be written as
$$M=\sum\limits_{s\in\mathbb Z}A_{n;t_0}^sL_s,$$
where $L_s$ is a Laurent polynomial in $\mathbb {KP}[A_{1;t_0}^{\pm 1},\ldots, A_{n-1;t_0}^{\pm 1}]$, i.e., $L_s$ does not contain any $A_{n;t_0}^{\pm 1}$.
By $M\neq 0$, we know that there exists some $s_0$ such that $L_{s_0}\neq 0$.  Let $t_1=\mu_{n}(t_0)$ and thus $t_0=\mu_n(t_1)$. Applying the exchange relation, we have
 $$A_{i;t_0}=A_{i;t_1}\;\;\text{and}\;\;\; A_{n;t_0}=A_{n;t_1}^{-1}P$$
where $i=1,\ldots,n-1$ and $P\in\mathbb {KP}[A_{1;t_1}^{\pm 1},\ldots, A_{n-1;t_1}^{\pm 1}]$ is the $n$-th exchange  binomial of the seed $t_1$. Now each $L_s$ can be also viewed as a Laurent polynomial in  $\mathbb {KP}[A_{1;t_1}^{\pm 1},\ldots, A_{n-1;t_1}^{\pm 1}]$. By our assumption in Subsection \ref{rmkexbi}, we know that $P$ is not invertible in $\mathbb {KP}[A_{1;t_1}^{\pm 1},\ldots, A_{n-1;t_1}^{\pm 1}]$.
Then by the fact that $\mathbb {KP}[A_{1;t_1}^{\pm 1},\ldots, A_{n-1;t_1}^{\pm 1}]$ is factorial and $L_{s_0}$ belongs to  $\mathbb {KP}[A_{1;t_1}^{\pm 1},\ldots, A_{n-1;t_1}^{\pm 1}]$, we know that there must exist some $m_0$ large enough such that $$L_{s_0}/P^{m_0}\notin \mathbb {KP}[A_{1;t_1}^{\pm 1},\ldots, A_{n-1;t_1}^{\pm 1}].$$ Set $m=m_0+s_0$. Then we know that
$$
M/A_{n;t_0}^m=\sum\limits_{s\in\mathbb Z}A_{n;t_0}^{s-m}L_s=\sum\limits_{s\in\mathbb Z}A_{n;t_1}^{m-s}(L_s/P^{m-s})\nonumber
$$
cannot be a Laurent polynomial in $\mathbb {KP}[A_{1;t_1}^{\pm1},\ldots,A_{n;t_1}^{\pm1}]$. So $M/A_{n;t_0}^m$ is not in $\mathcal U$. Then by (i), we know that  $(A_{k;t_0}\mid\mid  M)_v<m<\infty$.

(iii) This follows from (ii) and the fact $(A_{k;t_0}\mid\mid  0)_v=\infty$.

The results in (iv) and (v) follow from the definition of the valuation pairing.
\end{proof}

\begin{definition}[Local factorization] \label{deflocal}
Let $\mathcal U$ be an upper cluster algebra and $t$ a seed of $\mathcal U$. Let $M=N\cdot L$ be a factorization in $\mathcal U$. If $N$ is a cluster monomial in $t$ and $L\in\mathcal U$ satisfies $(A_{k;t}\mid\mid  L)_v=0$, that is, $L/A_{k;t}\notin \mathcal U$ for $k=1,\ldots,n$, we call $M=N\cdot L$ a {\em local factorization} of $M$ with respect to $t$.
\end{definition}

\begin{proposition}[Existence of local factorization] \label{proexist} Let $\mathcal U$ be an upper cluster algebra with coefficient semifield $\mathbb P$ and $t_0$ a seed of $\mathcal U$. Then any $0\neq M\in\mathcal U$ admits a local factorization with respect to $t_0$.
\end{proposition}
\begin{proof}
We first define a cluster monomial $N=A_{1;t_0}^{m_1}\cdots A_{n;t_0}^{m_n}$ in $t_0$.  Let $m_1:=(A_{1;t_0}\mid\mid  M)_v$ and  $m_2,\ldots,m_n$ are defined by induction. If $m_{k}$ has been defined, $m_{k+1}$ is defined as follows:
$$m_{k+1}:=(A_{k+1;t_0}\mid\mid  \frac{M}{A_{1;t_0}^{m_1}\cdots A_{k;t_0}^{m_k}})_v.$$
Thus the non-negative integers $m_1,\ldots,m_n$ are defined. So we get a cluster monomial $N=A_{1;t_0}^{m_1}\cdots A_{n;t_0}^{m_n}$ in $t_0$ and we know that $L:=M/N\in\mathcal U$.

Now we show that $L=M/N$ satisfies $(A_{k;t_0}\mid\mid  L)_v=0$, that is, $L/A_{k;t_0}\notin\mathcal U$ for $k=1,\ldots,n$. Assume by contradiction that there exists some $k\in\{1,\ldots,n\}$ such that
$$L/A_{k;t_0}=\frac{M}{A_{1;t_0}^{m_1}\cdots A_{k-1;t_0}^{m_{k-1}}A_{k;t_0}^{m_k+1}A_{k+1;t_0}^{m_{k+1}}\cdots A_{n;t_0}^{m_n}}\in\mathcal U.$$
Thus we get $$\frac{M}{A_{1;t_0}^{m_1}\cdots A_{k-1;t_0}^{m_{k-1}}A_{k;t_0}^{m_k+1}}=A_{k+1;t_0}^{m_{k+1}}\cdots A_{n;t_0}^{m_n}(L/A_{k;t_0})\in\mathcal U,$$
which contradicts the choice of $m_k$. So $L/A_{k;t_0}\notin\mathcal U$ for $k=1,\ldots,n$. Thus $M=N\cdot L$ is a local factorization of $M$ with respect to $t_0$.
This completes the proof.
\end{proof}

\begin{lemma}[Reduction Lemma]\label{lemadd}
Let $\mathcal U$ be a full rank upper cluster algebra and $t_0$ a seed of $\mathcal U$. Let $L$ be a non-zero element in $\mathcal U$ and $N=A_{1;t_0}^{m_1}\cdots A_{n;t_0}^{m_n}$ a cluster monomial in $t_0$. Then the following statements hold.
\begin{itemize}
\item[(i)] $L/A_{k;t_0}^s\in\mathcal U$ if and only if $L/A_{k;t_0}^s\in\mathcal L(t_k)$, where $s\in\mathbb N$, $t_k=\mu_k(t_0)$ and $$\mathcal L(t_k)=\mathbb{KP}[A_{1;t_k}^{\pm1},\ldots,A_{n;t_k}^{\pm1}].$$
\item[(ii)]  $(A_{k;t_0}\mid\mid  L)_v={\rm max}\{s\in\mathbb N\mid L/A_{k;t_0}^s\in\mathcal L(t_k)\}$.
\item[(iii)] $(A_{k;t_0}\mid\mid  N\cdot L)_v=m_k+(A_{k;t_0}\mid\mid  L)_v$.
\end{itemize}
\end{lemma}
\begin{proof}
(i)
By the Starfish Theorem \ref{proup}, we know that  $\mathcal U=\bigcap\limits_{i=0}^n\mathcal L(t_i)$. Since $A_{k;t_0}$ is invertible in $\mathcal L(t_i)$ for any $i\in\{0,1,\ldots,n\}\backslash\{k\}$, we have $L/A_{k;t_0}^s\in\mathcal L(t_i)$ for any $i\in\{0,1,\ldots,n\}\backslash\{k\}$. Thus we get that $L/A_{k;t_0}^s\in\mathcal U$ if and only if $L/A_{k;t_0}^s\in\mathcal L(t_k)$.

(ii) This follows from (i) and the definition of the valuation pairing.

(iii) Without loss of generality, we assume $k=1$. Since $A_{2;t_0}^{m_2}\cdots A_{n;t_0}^{m_n}$ is invertible in $\mathcal L(t_1)$ and by (ii), we have  $$(A_{1;t_0}\mid\mid  N\cdot L)_v=(A_{1;t_0}\mid\mid  \frac{N\cdot L}{A_{2;t_0}^{m_2}\cdots A_{n;t_0}^{m_n}})_v=(A_{1;t_0}\mid\mid  A_{1;t_0}^{m_1}\cdot L)_v.$$ Then by Proposition \ref{proind} (v), we get $$(A_{1;t_0}\mid\mid   A_{1;t_0}^{m_1}\cdot L)_v=m_1+(A_{1;t_0}\mid\mid  L)_v.$$
This completes the proof.
\end{proof}

\begin{theorem}[Local unique factorization property]\label{thmunique}
Let $\mathcal U$ be a full rank upper cluster algebra and $t_0$ a seed of $\mathcal U$. Then
  any $0\neq M\in\mathcal U$ admits a unique local factorization with respect to $t_0$.
\end{theorem}
\begin{proof}
The existence of local factorization of $M$ with respect to $t_0$ is known  from Proposition \ref{proexist}.
Now we show the uniqueness.
Let $M=N\cdot L$ be a  local factorization of $M$ with respect to $t_0$, where $N=A_{1;t_0}^{m_1}\cdots A_{n;t_0}^{m_n}$ is a cluster monomial in $t_0$ and $L\in\mathcal U$ satisfies $(A_{k;t_0}\mid\mid  L)_v=0$ for $k=1,\ldots,n$. By the Reduction Lemma \ref{lemadd} (iii), we have $$(A_{k;t_0}\mid\mid  M)_v=(A_{k;t_0}\mid\mid  N\cdot L)_v=m_k+(A_{k;t_0}\mid\mid  L)_v=m_k.$$
So $m_k=(A_{k;t_0}\mid\mid  M)_v$ is uniquely determined by $M$. Namely, $N$ is uniquely determined by $M$. This completes the proof.
\end{proof}
Now we give a counter-example to the local unique factorization property in the case where $\widetilde B_{t_0}$ is {\em not full rank}.

\begin{example}
Let $\mathcal U$ be a geometric upper cluster algebra with initial extended exchange matrix $\widetilde B_{t_0}$ given by $$\widetilde B_{t_0}= B_{t_0}=\begin{pmatrix}0&-1&0\\1&0&-1\\0&1&0\end{pmatrix}.$$
Clearly, $\widetilde B_{t_0}$ is not full rank.
Denote $t_1=\mu_1(t_0)$ and $t_3=\mu_3(t_0)$. Applying the exchange relations, we have the following equality.
\begin{eqnarray}\label{eqnexample}
A_{1;t_0}A_{1;t_1}=A_{2;t_0}+1=A_{3;t_0}A_{3;t_3}.
\end{eqnarray}
It is easy to check that both $A_{1;t_1}$ and $A_{3;t_3}$ are not divisible by any cluster variable in $t_0$ (alternatively, one can also refer to Theorem \ref{thmup} (ii)). So the equality (\ref{eqnexample}) gives two different local factorizations of $A_{2;t_0}+1$ with respect to the initial seed $t_0$.
\end{example}

\section{Application to factoriality of upper cluster algebras}\label{sec4}

In this section we give several equivalent characterizations for the factoriality of upper cluster algebras.  As an application, we show that full rank, primitive upper cluster algebras are factorial.

\subsection{Factoriality of partially compactified upper cluster algebras}\label{sec:40}

In this subsection, we reduce the factoriality of partially compactified upper cluster algebras to the factoriality of upper cluster algebras. 

Given a geometric upper cluster algebra $\mathcal U$,  let $\mathcal A$ be the corresponding cluster algebra. By definition, all the frozen variables are invertible in  $\mathcal U$ and $\mathcal A$. However, when studying the cluster structure on the coordinate rings of various algebraic varieties, it is important to allow that some of the frozen variables are not inverted.

We use ${\bf inv}$ to denote a subset of the set $\{Z_1,\cdots,Z_m\}$ of frozen variables. For convenience, we assume ${\bf inv}=\{Z_{p+1},\cdots,Z_m\}$. Denote by 
$${\mathcal L}(t,{\bf inv})=\mathbb K[Z_1,\cdots,Z_p,Z_{p+1}^{\pm1},\cdots,Z_m^{\pm1}][A_{1;t}^{\pm1},\cdots,A_{n;t}^{\pm1}]$$
the {\em partially compactified Laurent
polynomial ring}.

The {\em partially compactified upper cluster algebra} $\mathcal U({\bf inv})$ is defined to be
 $$\mathcal U({\bf inv})=\bigcap\limits_{t\in\mathbb T_n} {\mathcal L}(t,{\bf inv}).$$
 The {\em partially compactified cluster algebra} $\mathcal A({\bf inv})$ is defined to be
 $$\mathcal A({\bf inv})=\mathbb K[Z_1,\cdots,Z_p,Z_{p+1}^{\pm1},\cdots,Z_m^{\pm1}][A_{1;t},\cdots,A_{n;t}\mid t\in\mathbb T_n].$$
 Clearly, we have the inclusions $\mathcal U({\bf inv})\subseteq \mathcal U$ and $\mathcal A({\bf inv})\subseteq \mathcal A$. Moreover, the equality $\mathcal U({\bf inv})=\mathcal U$ (or $\mathcal A({\bf inv})=\mathcal A$) holds if and only if ${\bf inv}=\{Z_1,\cdots,Z_m\}$.

A domain $R$ is said to be {\em atomic} if every non-zero, non-invertible element $r\in R$   can be decomposed as a product of irreducible elements.

\begin{lemma}{\rm (}\cite[Corollary 1.20]{GLS19}{\rm ).} \label{lempp2}
Let $R$ be an atomic domain and $S\subseteq R\setminus\{0\}$ a multiplicative set generated by prime elements. Then $R$ is factorial if and only if the localization $S^{-1}R$ is.
\end{lemma}

\begin{remark}\label{rmk:atomic}
In the above lemma, the assumption that $R$ is atomic is missing in the original statement of \cite[Corollary 1.20]{GLS19}. We thank 
the first and third authors of \cite{GLS19} for pointing this out in an email message \cite{GS_2022}. In the same email message, they also included a proof that any geometric (upper) 
cluster algebra is an atomic domain.
Their proof actually also works for partially compactified (upper) cluster algebras.
\end{remark}

\begin{proposition}[\cite{GS_2022}]
   \label{pro:atomic}
  $\mathcal U({\bf inv})$ and $\mathcal A({\bf inv})$ are atomic.
\end{proposition}
The proof of the above proposition is put in the appendix. We remark that the main idea for the proof follows that of \cite{GS_2022} with some modifications and simplifications.

\begin{lemma}\label{lem:l-inv}
 Let $M$ be an element in $\mathcal U$. Then $M$ belongs to $\mathcal U({\bf inv})$ if and only if $M$ belongs to $\mathcal L(t,{\bf inv})$ for some seed $t$.
\end{lemma}
 \begin{proof}
 ``$\Longrightarrow$": This is clear.

 ``$\Longleftarrow$": Assume that $M\in\mathcal U$ belongs to $\mathcal L(t,{\bf inv})$ for some seed $t$.
Using similar arguments as in the proof of Proposition \ref{prodvectors}, we get $M\in \mathcal L(\mu_k(t),{\bf inv})$, where $k=1,\ldots,n$. Then by induction, we have $M\in \mathcal L(t',{\bf inv})$ for any seed $t'$ of $\mathcal U({\bf inv})$. Hence, we have $M\in\mathcal U({\bf inv})$.
 \end{proof}

 \begin{lemma}{\rm (}\cite[Lemma 7.2]{GLS19}, \cite[Proposition 4.3(i)]{MNTY-2022}{\rm ).}\label{lempp3} Keep ${\bf inv}=\{Z_{p+1},\cdots,Z_m\}$.
The non-invertible frozen variables $Z_1,\ldots,Z_p$ are prime elements in $\mathcal U({\bf inv})$.
 \end{lemma}

\begin{proposition}\label{prop:compactified_U_factoriality}
 $\mathcal U({\bf inv})$ is factorial if and only if $\mathcal U$ is factorial.
\end{proposition}
\begin{proof}
By Lemma \ref{lempp3}, we know that the non-invertible frozen variables $Z_1,\ldots,Z_p$ are prime elements in $\mathcal U({\bf inv})$.
    Let $S$ be the multiplicative set in  $\mathcal U({\bf inv})$ generated by the prime elements $Z_1,\cdots,Z_p$.
    Clearly, we have the inclusion $S^{-1}(\mathcal U({\bf inv}))\subseteq \mathcal U$. Using  Lemma \ref{lem:l-inv}, one can easily show the converse inclusion. Hence, we have  $S^{-1}(\mathcal U({\bf inv}))=\mathcal U$. Then the result follows from Proposition \ref{pro:atomic} and Lemma \ref{lempp2}. 
\end{proof}

Thanks to the above proposition, we will mainly focus on the study the factoriality of $\mathcal U$.

\subsection{Characterizations for the factoriality of upper cluster algebras}\label{sec41}
We first give an observation which shows why the valuation pairing can be used to study the  factoriality of upper cluster algebras.

\begin{proposition}[Observation]\label{observation}
Let $\mathcal U$ be a geometric upper cluster algebra  and $A_{k;t}$ a cluster variable of $\mathcal U$. Then the following two statements are equivalent.
\begin{itemize}
\item[(i)] $A_{k;t}$ is prime in $\mathcal U$.
\item[(ii)] For any non-zero elements $M$ and $L$ in $\mathcal U$, we have the following equality:
\begin{eqnarray}\label{eqnmain}
(A_{k;t}\mid\mid M\cdot L)_v=(A_{k;t}\mid\mid M)_v+(A_{k;t}\mid\mid L)_v.
\end{eqnarray}
\end{itemize}
\end{proposition}
\begin{proof}
(i)$\Rightarrow$ (ii): This is clear from the definition of the valuation pairing.

(ii)$\Rightarrow$ (i):
For any $0\neq M,L\in\mathcal U$ with $A_{k;t}\mid M\cdot L$ in $\mathcal U$, we claim  that either $A_{k;t}\mid  M$ in $\mathcal U$ or $A_{k;t}\mid L$ in $\mathcal U$. Otherwise, we have that
both $M$ and $L$ are not divisible by $A_{k;t}$ in $\mathcal U$. Then we have $$(A_{k;t}\mid\mid M)_v=0=(A_{k;t}\mid\mid L)_v.$$
 Now by our assumption, we have $$(A_{k;t}\mid\mid M\cdot L)_v=(A_{k;t}\mid\mid M)_v+(A_{k;t}\mid\mid L)_v=0.$$
This contradicts $A_{k;t}\mid M\cdot L$ in $\mathcal U$. So either $A_{k;t}\mid  M$ or $A_{k;t}\mid L$ in $\mathcal U$.

By Theorem \ref{thmgls} (ii), we know that $A_{k;t}$ is not invertible in $\mathcal U$. Hence,  $A_{k;t}$ is prime in $\mathcal U$.
\end{proof}

It is natural to ask under which conditions the equality (\ref{eqnmain}) always holds. Thanks to Reduction Lemma \ref{lemadd}, we can answer this question for {\em full rank upper cluster algebras}.

\begin{proposition}\label{prozp}
Let $\mathcal U$ be a full rank upper cluster algebra and $t$ a seed of $\mathcal U$. Put $t_k=\mu_k(t)$ and let $P_{k;t}$ be the $k$-th exchange binomial of $t$. Then the following statements are equivalent.
 \begin{itemize}
\item[(i)]  $(A_{k;t}\mid\mid M\cdot L)_v=(A_{k;t}\mid\mid M)_v+(A_{k;t}\mid\mid L)_v$ holds for any non-zero elements $M$ and $L$ in $\mathcal U$.
\item[(ii)] $A_{k;t}$ is  prime in $\mathcal L(t_k)=\mathbb {K}[Z_1^{\pm1},\ldots,Z_m^{\pm1}][A_{1;t_k}^{\pm1},\ldots,A_{n;t_k}^{\pm1}]$.
\item[(iii)] $P_{k;t}$ is prime in $\mathcal L(t_k)$.
\item[(iv)] $P_{k;t}$ is prime  in $\mathbb K[Z_1,\ldots,Z_m][A_{1;t},\ldots,\widehat A_{k;t},\ldots,A_{n;t}]$, which is actually equal to $\mathbb K[Z_1,\ldots,Z_m][A_{1;t_k},\ldots,\widehat A_{k;t_k},\ldots,A_{n;t_k}]$.
    \item [(v)] $P_{k;t}$ is irreducible  in  $\mathbb K[Z_1,\ldots,Z_m][A_{1;t},\ldots,\widehat A_{k;t},\ldots,A_{n;t}]$.
\end{itemize}
\end{proposition}
\begin{proof}
(i)$\Longrightarrow$ (ii): By our assumption in Subsection \ref{rmkexbi}, we know that $A_{k;t}$ is not invertible in $\mathcal L(t_k)$. Now assume by contradiction that $A_{k;t}$ is not prime in $\mathcal L(t_k)$. Then there exists $M^\prime,L^\prime\in \mathcal L(t_k)$ such that both $M^\prime$ and $L^\prime$ are not divisible by $A_{k;t}$ in $\mathcal L(t_k)$, but the product $M^\prime\cdot L^\prime$ is divisible by  $A_{k;t}$ in $\mathcal L(t_k)$.  Since $M^\prime$ and $L^\prime$ are in $$\mathcal L(t_k)=\mathbb {K}[Z_1^{\pm1},\ldots,Z_m^{\pm1}][A_{1;t_k}^{\pm1},\ldots,A_{n;t_k}^{\pm1}],$$ there exists a cluster monomial
$A_{t_k}^{\bf v}$ in $t_k$ such that $M:=A_{t_k}^{\bf v}M^\prime$ and $L:=A_{t_k}^{\bf v} L^\prime$ belong to $\mathcal U$.

Because $A_{t_k}^{\bf v}$ is invertible in $\mathcal L(t_k)$, we know that both  $M$ and $L$ are not divisible by $A_{k;t}$ in $\mathcal L(t_k)$, but the product $M\cdot L$ is divisible by  $A_{k;t}$ in $\mathcal L(t_k)$. By the Reduction Lemma \ref{lemadd} (ii), we have $$(A_{k;t}\mid\mid  M)_v=0=(A_{k;t}\mid\mid L)_v\;\; \text{but}\;\;(A_{k;t}\mid\mid  M\cdot L)_v>0.$$ This contradicts that $(A_{k;t}\mid\mid M\cdot L)_v=(A_{k;t}\mid\mid M)_v+(A_{k;t}\mid\mid L)_v$ holds for any non-zero elements $M$ and $L$ in $\mathcal U$.
Hence,  $A_{k;t}$ is  prime in $\mathcal L(t_k)$.

(ii)$\Longrightarrow$ (i): This follows from the Reduction Lemma \ref{lemadd} (ii).

(ii)$\Longleftrightarrow$ (iii): Applying the exchange relation, we have $A_{k;t}A_{k;t_k}=P_{k;t}$. Since $A_{k;t_k}$ is invertible in $\mathcal L(t_k)$, we know that $A_{k;t}$ is prime in $\mathcal L(t_k)$ if and only if $P_{k;t}$ is prime in $\mathcal L(t_k)$.

(iii)$\Longleftrightarrow$ (iv): Notice that $\mathcal L(t_k)$ is a localization of $\mathbb K[Z_1,\ldots,Z_m][A_{1;t_k},\ldots,A_{n;t_k}]$ at $Z_1,\ldots,Z_m, A_{1;t_k},\ldots,A_{n;t_k}$. Since $P_{k;t}=P_{k;t_k}$ is a binomial in $$\mathbb K[Z_1,\ldots,Z_m][A_{1;t_k},\ldots,A_{n;t_k}]$$ and it is not divisible by any $Z_i$ and $A_{j;t_k}$, we know that $P_{k;t}$ is prime in  $\mathcal L(t_k)$ if and only if it is prime in $\mathbb K[Z_1,\ldots,Z_m][A_{1;t_k},\ldots,A_{n;t_k}]$. Because $P_{k;t}$ is actually a binomial in variables from
$$\{Z_1,\ldots,Z_m, A_{1;t_k},\ldots,\widehat A_{k;t_k},\ldots,A_{n;t_k}\}=\{Z_1,\ldots,Z_m, A_{1;t},\ldots,\widehat A_{k;t},\ldots,A_{n;t}\},$$
 we know that $P_{k;t}$ is prime in $\mathbb K[Z_1,\ldots,Z_m][A_{1;t_k},\ldots,A_{n;t_k}]$ if and only if it is prime in $$\mathbb K[Z_1,\ldots,Z_m][A_{1;t},\ldots,\widehat A_{k;t},\ldots,A_{n;t}].$$
Hence,  $P_{k;t}$ is prime in  $\mathcal L(t_k)$ if and only if it is prime in $$\mathbb K[Z_1,\ldots,Z_m][A_{1;t},\ldots,\widehat A_{k;t},\ldots,A_{n;t}].$$

(iv)$\Longleftrightarrow$ (v): This follows from the fact that the polynomial ring  $$\mathbb K[Z_1,\ldots,Z_m][A_{1;t},\ldots,\widehat A_{k;t},\ldots,A_{n;t}]$$ is factorial.
\end{proof}

\begin{theorem}\label{thmcufd}
Let $\mathcal U$ be a full rank upper cluster algebra with initial seed $t_0$. Then the following statements are equivalent.
\begin{itemize}
\item[(i)] $\mathcal U$ is factorial.
 \item[(ii)] Any initial  cluster variable  $A_{k;t_0}$ is prime in $\mathcal U$.
\item[(iii)] Any initial cluster variable $A_{k;t_0}$  satisfies
$$(A_{k;t_0}\mid\mid M\cdot L)_v=(A_{k;t_0}\mid\mid M)_v+(A_{k;t_0}\mid\mid L)_v,$$
for any non-zero elements $M$ and $L$ in $\mathcal U$.
\item[(iv)]  Any exchange binomial $P_{k;t_0}$ of $t_0$ is irreducible in $$\mathbb K[Z_1,\ldots,Z_m][A_{1;t_0},\ldots,\widehat A_{k;t_0},\ldots,A_{n;t_0}].$$
\end{itemize}
\end{theorem}
\begin{proof}

(i)$\Longrightarrow$ (ii): By Theorem \ref{thmgls} (ii), we know that any cluster variable is irreducible in $\mathcal U$. Because $\mathcal U$ is factorial, we get that any irreducible element is prime in $\mathcal U$. In particular, any initial  cluster variable  $A_{k;t_0}$ is prime in $\mathcal U$.

(ii)$\Longrightarrow$ (i): Let $S$ be the  multiplicative set generated by initial cluster variables. By the Laurent phenomenon, we have $$S^{-1}\mathcal U=\mathcal L(t_0)=\mathbb K[Z_1^{\pm1},\ldots,Z_m^{\pm1}][A_{1;t_0}^{\pm1},\ldots,A_{n;t_0}^{\pm1}],$$
which is factorial.
Since the initial cluster variables are prime in $\mathcal U$ and by Lemma \ref{lempp2} and Proposition \ref{pro:atomic}, we get that $\mathcal U$ is factorial.

The equivalence of (ii), (iii) and (iv) follows from Proposition \ref{prozp}.
\end{proof}

\begin{remark}
Garcia Elsener {\em et al.} in \cite[Theorem 5.1]{GLS19} prove the equivalence of (i) and (iv) in Theorem \ref{thmcufd} for  geometric cluster algebra $\mathcal A$ (possibly not full rank) with an {\em acyclic initial seed}.  Note that in acyclic case, $\mathcal A=\mathcal U$, by Proposition \ref{proacyclic}. Our method here is very different with that in \cite{GLS19}.
\end{remark}

\subsection{Full rank, primitive upper cluster algebras are factorial}\label{sec42}
In this subsection, we show that  full rank, primitive upper cluster algebras are factorial.

Recall that a vector $(b_1,\ldots,b_{n+m})^{\rm T}\in\mathbb Z^{n+m}$ is primitive if the greatest common divisor $d\in\mathbb N$ of $b_1,\ldots,b_{n+m}$ is $1$. An extended exchange matrix $\widetilde B_t$ is {\em primitive} if each column vector of  $\widetilde B_t$ is primitive.

\begin{definition}[Primitive upper cluster algebra] Let $\mathcal U$ be a geometric upper cluster algebra with initial seed $t_0$. We say that $\mathcal U$ is a {\em primitive upper cluster algebra}, if the initial extended exchange matrix $\widetilde B_{t_0}$ is primitive.
\end{definition}

\begin{proposition}\label{propriex}
Let $\mathcal U$ be a primitive upper cluster algebra. Then for any seed $t$, the following statements hold.
\begin{itemize}
\item[(i)]{\rm (}\cite[Lemma 4.3]{LP16}{\rm ).}  $\widetilde B_t$ is primitive.
\item[(ii)] {\rm (}\cite[Lemma 4.1]{LP16}{\rm ).} The $k$-th exchange binomial of $t$ is irreducible in $$\mathbb K[Z_1,\ldots,Z_m][A_{1;t},\ldots,\widehat A_{k;t},\ldots,A_{n;t}],$$ where $k=1,\ldots,n$.
\end{itemize}
\end{proposition}

\begin{theorem}\label{thmufd}
Let $\mathcal U$ be a full rank, primitive upper cluster algebra with initial seed $t_0$, and let $$M=\frac{P_M(A_{1;t_0},\ldots,A_{n;t_0})}{A_{1;t_0}^{d_1}\cdots A_{n;t_0}^{d_n}}$$ be a non-initial cluster variable of $\mathcal U$,  where ${\bf d}:=(d_1,\ldots,d_n)^{\rm T}$ is the $d$-vector of $M$ with respect to $t_0$. Then
\begin{itemize}
\item[(i)] $\mathcal U$ is factorial.
\item[(ii)] the numerator polynomial $P_M$ of $M$ is irreducible in $$\mathbb K[Z_1,\ldots,Z_m][A_{1;t_0},\ldots,A_{n;t_0}].$$
\end{itemize}

\end{theorem}
\begin{proof}
(i)  By Proposition \ref{propriex} (ii), we know that every exchange binomial of $t_0$ is irreducible in the corresponding polynomial ring. Then the result follows from Theorem \ref{thmcufd} (i)(iv).

(ii)  We first show that  $P_M$ is not invertible in $\mathbb K[Z_1,\ldots,Z_m][A_{1;t_0},\ldots,A_{n;t_0}]$.
Since $M$ is a non-initial cluster variable and by Proposition \ref{rmkdf} (4), we have ${\bf d}\in\mathbb N^n$ and thus $A_{t_0}^{\bf d}\in\mathcal U$. By $M=P_M/A_{t_0}^{\bf d}$, we get $MA_{t_0}^{\bf d}=P_M$. By Theorem \ref{thmgls} (2), we know that both $M$ and $A_{t_0}^{\bf d}$ are not invertible in $\mathcal U$. So $P_M=MA_{t_0}^{\bf d}$ is not invertible in $\mathcal U$. Namely, we have
$$P_M\notin\mathcal U^\times=\{\lambda Z_1^{c_1}\cdots Z_m^{c_m}\mid \lambda\in\mathbb K^\times,\; c_1,\ldots,c_m\in \mathbb Z\}.$$
In particular, $P_M\notin\mathbb K^\times$. So $P_M$ is not invertible in $\mathbb K[Z_1,\ldots,Z_m][A_{1;t_0},\ldots,A_{n;t_0}]$.

We claim that $P_M$ is not divisible by any $Z_i$ and $A_{j;t_0}$ in $$\mathbb K[Z_1,\ldots,Z_m][A_{1;t_0},\ldots,A_{n;t_0}].$$ Since $P_M$ is a numerator polynomial, it is not divisible by any $A_{j;t_0}$. There are two approaches to prove $Z_i\nmid P_M$. The first one is that we prove the result by induction on the minimal length of a sequence of  mutations from the initial seed to the final seed. The second one is that we just view $Z_i$ as an exchangeable cluster variable. Then by the fact that $M$ and $Z_i$ are in the same cluster and  they are not equal, we get $(Z_i\mid\mid M)_d=0$, by Proposition \ref{rmkdf} (5). This implies that $P_M$ is not divisible by $Z_i$.

We claim that  $P_M$ is not invertible in
$$\mathcal L(t_0)=\mathbb K[Z_1^{\pm1},\ldots,Z_m^{\pm1}][A_{1;t_0}^{\pm1},\ldots,A_{n;t_0}^{\pm1}].$$ This follows from the two facts that $P_M$ is not invertible in $$\mathbb K[Z_1,\ldots,Z_m][A_{1;t_0},\ldots,A_{n;t_0}]$$ and $P_M$ is not divisible by any $Z_i$ and $A_{j;t_0}$ in $\mathbb K[Z_1,\ldots,Z_m][A_{1;t_0},\ldots,A_{n;t_0}]$.

Now we show that $P_M$ is irreducible in $\mathcal L(t_0)$. Assume that $P_M=L_1L_2$ for some $L_1,L_2\in\mathcal L(t_0)$. We know that there exist vectors ${\bf v}_1,{\bf v}_2\in\mathbb N^n$ such that
$$U_1:=A_{t_0}^{{\bf v}_1}L_1\in\mathcal U\;\;\text{and}\;\;\; U_2:=A_{t_0}^{{\bf v}_2}L_2\in\mathcal U.$$ By $M=\frac{P_M}{A_{t_0}^{\bf d}}=\frac{L_1L_2}{A_{t_0}^{\bf d}}=\frac{U_1U_2}{A_{t_0}^{{\bf d}+{\bf v}_1+{\bf v}_2}}$, we get
\begin{eqnarray}\label{eqnmau}
MA_{t_0}^{{\bf d}+{\bf v}_1+{\bf v}_2}=U_1U_2,
\end{eqnarray}
which is an equality in $\mathcal U$, since ${\bf d},{\bf v}_1,{\bf v}_2\in\mathbb N^n$. Because $\mathcal U$ is factorial and $M$ is irreducible in $\mathcal U$, we get that either $U_1$ or $U_2$ is divisible by $M$ in $\mathcal U$. Without loss of generality, we assume that $U_1^\prime:=U_1/M\in\mathcal U$. Then by the equality (\ref{eqnmau}), we get $U_1^\prime U_2=A_{t_0}^{{\bf d}+{\bf v}_1+{\bf v}_2},$ which implies that both $U_1^\prime$ and $U_2$ are invertible in $\mathcal L(t_0)$. So $L_2=U_2/A_{t_0}^{{\bf v}_2}$ is invertible in $\mathcal L(t_0)$. Hence,  $P_M$ is irreducible in $\mathcal L(t_0)$.

Since $P_M$ is irreducible in $\mathcal L(t_0)$ and by
the fact that $P_M$ is not divisible by any $Z_i$ and $A_{j;t_0}$ in $\mathbb K[Z_1,\ldots,Z_m][A_{1;t_0},\ldots,A_{n;t_0}],$ we get that $P_M$ is irreducible in $\mathbb K[Z_1,\ldots,Z_m][A_{1;t_0},\ldots,A_{n;t_0}]$. This completes the proof.
\end{proof}

\begin{remark}(i) Note that any principal coefficient upper cluster algebra is primitive and has full rank. So principal coefficient upper cluster algebras are always factorial.

(ii) A different proof of Theorem \ref{thmufd} (i) via the factoriality of Cox rings can be found in \cite[Corollary 4.7]{GHK15}.

 (iii) Garcia Elsener {\em et al.} in \cite[Corollary 5.3]{GLS19} prove that any principal coefficient cluster algebra $\mathcal A$ with an {\em acyclic initial exchange matrix} is factorial.
 Note that in the acyclic case, we have $\mathcal A=\mathcal U$, by Proposition \ref{proacyclic}.
\end{remark}

\begin{corollary}\label{corzp}
Let $\mathcal U$ be a full rank, primitive upper cluster algebra and $A_{k;t_0}$ a cluster variable of $\mathcal U$. Then for any non-zero elements $M$ and $L$ in $\mathcal U$, we have
$$(A_{k;t_0}\mid\mid M\cdot L)_v=(A_{k;t_0}\mid\mid M)_v+(A_{k;t_0}\mid\mid L)_v.$$
\end{corollary}
\begin{proof}
This follows from Theorem \ref{thmcufd} (i)(iii) and Theorem \ref{thmufd} (i).
\end{proof}

\subsection{Examples of non-factorial upper cluster algebras}\label{sec43}

\begin{proposition}{\rm (}\cite[Corollary 6.5]{GLS13}{\rm ).}
\label{corgls}
Let $\mathcal A$ be a geometric cluster algebra with initial extended exchange matrix $$\widetilde B_{t_0}= B_{t_0}=\begin{pmatrix}0&-c\\d&0\end{pmatrix},$$
where $c\geq1$ and $d\geq3$ an odd number. Then $\mathcal A$ is not factorial.
\end{proposition}

\begin{remark}\label{rmknonfactorial}
Note that in Proposition \ref{corgls}, $\mathcal A$ is an acyclic cluster algebra. In this case, $\mathcal A$ coincides with its upper cluster algebra $\mathcal U$, by Proposition \ref{proacyclic}. So Proposition \ref{corgls} provides many upper cluster algebras which have full rank but are not factorial.
\end{remark}

Let $\mathcal U$ be the  geometric upper cluster algebra with initial extended exchange matrix $$\widetilde B_{t_0}= B_{t_0}=\begin{pmatrix}0&-1\\3&0\end{pmatrix}.$$
We know that $\mathcal U$ is a full rank cluster algebra.  It is non-factorial, by Remark \ref{rmknonfactorial}. From the viewpoint of Theorem \ref{thmcufd}, the non-factoriality of $\mathcal U$ is due to the fact that the first exchange binomial $A_{2;t_0}^3+1=(A_{2;t_0}+1)(A_{2;t_0}^2-A_{2;t_0}+1)$ of $t_0$ is not irreducible.

The following is a concrete example where $\mathcal U$ is primitive but not factorial.

\begin{example}
Let $\mathcal A$ be a geometric cluster algebra with the following initial extended exchange matrix $$\widetilde B_{t_0}= B_{t_0}=\begin{pmatrix}0&-1&0\\1&0&-1\\0&1&0\end{pmatrix}.$$
It is easy to see that $\widetilde B_{t_0}$ is primitive and acyclic but not full rank.
It is known from \cite[Proposition 6.1]{GLS13} that $\mathcal A$ is not factorial. Since $\mathcal A$ is acyclic, $\mathcal A$ coincides with its upper cluster algebra $\mathcal U$, by Proposition \ref{proacyclic}. So $\mathcal U=\mathcal A$ is not factorial.
\end{example}

\subsection{Ray Fish Theorem}\label{sec44}
Recall that the Starfish Theorem states that any full rank upper cluster algebra $\mathcal U$ can be written as the intersection of $n+1$ Laurent polynomial rings. In this subsection, we show that any full rank, primitive upper cluster algebra can be written as the intersection of {\em two} Laurent polynomial rings. We call this the Ray Fish Theorem.

\begin{proposition} {\rm (}\cite[Corollary 4.2]{GLS13}{\rm ).} \label{proglsau} Let $\mathcal A$ be a geometric cluster algebra with coefficient semifield  $\mathbb P={\rm Trop}(Z_1,\ldots,Z_m)$  and $\mathcal U$ its upper cluster algebra. Let $t, t_0$ be two seeds of $\mathcal A$ with no common cluster variables. If $\mathcal A$ is factorial, then $$\mathcal A=\mathcal U=\mathcal L(t)\cap\mathcal L(t_0),$$ where $\mathcal L(t)=\mathbb {KP}[A_{1;t}^{\pm1},\ldots,A_{n;t}^{\pm1}]$ and $\mathcal L(t_0)=\mathbb {KP}[A_{1;t_0}^{\pm1},\ldots,A_{n;t_0}^{\pm1}]$.
\end{proposition}
\begin{remark}
By Proposition \ref{proglsau}, a factorial geometric cluster algebra must coincide with its upper cluster algebra. This provides a reason why we focus on the factoriality of upper cluster algebras rather than cluster algebras.
\end{remark}
\begin{corollary}
Let $\mathcal U$ be a full rank, primitive upper cluster algebra and $\mathcal A$ the corresponding cluster algebra. Then $\mathcal A=\mathcal U$ if and only if $\mathcal A$ is factorial.
\end{corollary}
\begin{proof}
``$\Longrightarrow$": Since $\mathcal U$ is full rank and primitive and by Theorem \ref{thmufd}, we know that $\mathcal U$ is factorial. Then the factoriality of $\mathcal A$ follows from  $\mathcal A=\mathcal U$.

``$\Longleftarrow$": This follows from Proposition \ref{proglsau}.
\end{proof}
\begin{lemma}\label{lemassociate}
Let $\mathcal U$ be a geometric upper cluster algebra with coefficient semifield $\mathbb P={\rm Trop}(Z_1,\ldots,Z_m)$ and $A_{i;t_0}, A_{j;t}$ two cluster variables of $\mathcal U$. Then $A_{i;t_0}$ and $A_{j;t}$ are associate in $\mathcal U$ if and only if $A_{i;t_0}=A_{j;t}$.
\end{lemma}
\begin{proof}
Assume that  $A_{i;t_0}$ and $A_{j;t}$ are associate in $\mathcal U$. Then
\begin{eqnarray}\label{eqnassociate}
A_{j;t}=\lambda Z_1^{c_1}\cdots Z_m^{c_m}A_{i;t_0}
 \end{eqnarray}
 for some $\lambda\in\mathbb K^\times$ and  $c_1,\ldots,c_m\in \mathbb Z$, by Theorem \ref{thmgls} (ii). Notice that the equality (\ref{eqnassociate}) can be viewed as the Laurent expansion of $A_{j;t}$ with respect to $t_0$. It is easy to see that $(A_{i;t_0}\mid\mid A_{j;t})_d=-1$. Then by Proposition \ref{rmkdf} (4), we get $A_{j;t}=A_{i;t_0}$.

Conversely, assume $A_{i;t_0}=A_{j;t}$. Clearly, we have that   $A_{i;t_0}$ and $A_{j;t}$ are associate in $\mathcal U$. This completes the proof.
\end{proof}

Inspired by Proposition \ref{proglsau}, we give the following result.
\begin{theorem}[Ray Fish Theorem]\label{thmstrong}
Let $\mathcal U$ be a full rank, primitive upper cluster algebra with coefficient semifield $\mathbb P={\rm Trop}(Z_1,\ldots,Z_m)$ and $t, t_0$ two seeds of $\mathcal U$ with no common cluster variables. Then $\mathcal U=\mathcal L(t)\cap\mathcal L(t_0)$, where $\mathcal L(t)=\mathbb {KP}[A_{1;t}^{\pm1},\ldots,A_{n;t}^{\pm1}]$ and $\mathcal L(t_0)=\mathbb {KP}[A_{1;t_0}^{\pm1},\ldots,A_{n;t_0}^{\pm1}]$.
\end{theorem}

\begin{proof}The proof is the same as that of \cite[Corollary 4.2]{GLS13}. For convenience of the reader, we repeat it here.

Since the inclusion $\mathcal U\subseteq \mathcal L(t)\cap\mathcal L(t_0)$ is clear, it suffices to show the converse inclusion.
We know that
any $M\in \mathcal L(t)\cap\mathcal L(t_0)$ has the following form:
$$M=\frac{P(A_{1;t},\ldots,A_{n;t})}{A_{1;t}^{v_1}\cdots A_{n;t}^{v_n}}=\frac{Q(A_{1;t_0},\ldots,A_{n;t_0})}{A_{1;t_0}^{s_1}\cdots A_{n;t_0}^{s_n}},$$
where $P\in\mathbb{KP}[A_{1;t},\ldots,A_{n;t}],\; Q\in\mathbb{KP}[A_{1;t_0},\ldots,A_{n;t_0}]$ and $v_i,s_i\geq0$ for any $i=1,\ldots,n$. Thus we have $A_{1;t_0}^{s_1}\cdots A_{n;t_0}^{s_n}P=QA_{1;t}^{v_1}\cdots A_{n;t}^{v_n}.$

By Theorem \ref{thmgls} (ii), we know that $A_{i;t_0}$ and $A_{j;t}$ are irreducible in $\mathcal U$ for any $i$ and $j$. By Lemma \ref{lemassociate} and the fact that $t$ and $t_0$ have no common cluster variables, we know that $A_{i;t_0}$ and $A_{j;t}$ are non-associate for all $1\leq i,j\leq n$. Since $\mathcal U$ is full rank and primitive and by Theorem \ref{thmufd}, we know that $\mathcal U$ is factorial. So the equality $A_{1;t_0}^{s_1}\cdots A_{n;t_0}^{s_n}P=QA_{1;t}^{v_1}\cdots A_{n;t}^{v_n}$ implies that $P$ is divisible by $A_{1;t}^{v_1}\cdots A_{n;t}^{v_n}$ in $\mathcal U$. Namely, there exists $P^\prime\in\mathcal U$ such that $P= A_{1;t}^{v_1}\cdots A_{n;t}^{v_n}P^\prime$. Thus $$M=\frac{P}{A_{1;t}^{v_1}\cdots A_{n;t}^{v_n}}=P^\prime\in\mathcal U.$$
So we have $\mathcal L(t)\cap\mathcal L(t_0)\subseteq\mathcal U$ and thus $\mathcal U=\mathcal L(t)\cap\mathcal L(t_0)$.
\end{proof}

\section{Applications to \texorpdfstring{$d$}{Lg}-vectors}\label{sec5}

\subsection{\texorpdfstring{$d$}{Lg}-compatibility degree and \texorpdfstring{$d$}{Lg}-vectors via the valuation pairing}\label{sec51}
In this subsection, we show how to express the $d$-compatibility and the $d$-vectors using the valuation pairing for full rank upper cluster algebras.

\begin{theorem}\label{thmind}
 Let $\mathcal U$ be a full rank upper cluster algebra with initial seed $t_0$. Let
 $$M=\frac{P_M(A_{1;t_0},\ldots,A_{n;t_0})}{A_{1;t_0}^{d_1}\cdots A_{n;t_0}^{d_n}}$$ be a non-zero element in $\mathcal U$,  where ${\bf d}_M:=(d_1,\ldots,d_n)^{\rm T}$ is the $d$-vector of $M$ with respect to $A_{t_0}$. Then the following statements hold.
 \begin{itemize}
 \item[(i)] For any $k=1,\ldots,n$, we have $$d_k=(A_{k;t_0}\mid\mid M)_d=(A_{k;t_0}\mid\mid P_M)_v-(A_{k;t_0}\mid\mid M)_v.$$
 \item[(ii)] If $(A_{k;t_0}\mid\mid M)_v=0$ for some $k$, that is, $M/A_{k;t_0}\notin\mathcal U$, then $$d_k=(A_{k;t_0}\mid\mid M)_d=(A_{k;t_0}\mid\mid P_M)_v\geq 0.$$
 \item[(iii)] If $(A_{k;t_0}\mid\mid M)_v=0$ for $k=1,\ldots,n$, then ${\bf d}_M$ is uniquely determined by $P_M$.
 \end{itemize}

\end{theorem}
\begin{proof}(i) We know that there exist two vectors $${\bf d}^+=(d_1^+,\ldots,d_n^+)^{\rm T},\;{\bf d}^-=(d_1^-,\ldots,d_n^-)^{\rm T}\in\mathbb N^n$$ such that $${\bf d}_M={\bf d}^+-{\bf d}^-.$$
By $M=P_M/A_{t_0}^{{\bf d}_M}=P_M/A_{t_0}^{{\bf d}^+-{\bf d}^-}$, we have $A_{t_0}^{{\bf d}^+}M=A_{t_0}^{{\bf d}^-}P_M$.
By the Reduction Lemma \ref{lemadd} (iii), we get

$$(A_{k;t_0}\mid\mid M)_v+d_k^+=(A_{k;t_0}\mid\mid A_{t_0}^{{\bf d}_+}M)_v=(A_{k;t_0}\mid\mid A_{t_0}^{{\bf d}_-}P_M)_v=(A_{k;t_0}\mid\mid P_M)_v+d_k^-.$$
So we have $$d_k=d_k^+-d_k^-=(A_{k;t_0}\mid\mid P_M)_v-(A_{k;t_0}\mid\mid M)_v.$$

The results in (ii) and (iii) follow from (i).
\end{proof}

\subsection{Local factorizations of cluster monomials}\label{sec52}
Let $\mathcal U$ be an upper cluster algebra with coefficient semifield  $\mathbb P$ and   initial seed $t_0$. In this subsection, we prove that if $M$ is a cluster monomial in non-initial cluster variables, then $M/A_{k;t_0}\notin\mathcal U$, that is, $(A_{k;t_0}\mid\mid M)_v=0$ for $k=1,\ldots,n$. In particular, we can give a local factorization for any cluster monomial.

\begin{lemma}\label{lemmonomial}
Let $\mathcal U$ be an upper cluster algebra with coefficient semifield $\mathbb P$  and $t$ a seed of $\mathcal U$. If $N=A_{1;t}^{b_1}A_{2;t}^{b_2}\cdots A_{n;t}^{b_n}\in\mathcal U$ for some $b_1, b_2,\ldots,b_n\in\mathbb Z$, then $b_1, b_2,\ldots,b_n\geq 0$.
\end{lemma}
\begin{proof}
Assume by contradiction that there exists some $b_{i_0}<0$. Without loss of generality, we just assume $b_1<0$.
Let $t^\prime=\mu_{1}(t)$ and thus $t=\mu_1(t^\prime)$. Applying the exchange relation, we know that
$$A_{1;t}A_{1;t^\prime}=P\in\mathbb{KP}[A_{2;t^\prime},\ldots,A_{n;t^\prime}],$$ where $P$ is the first exchange binomial of $t^\prime$. Then we know that
\begin{eqnarray}
N=A_{1;t}^{b_1}A_{2;t}^{b_2}\cdots A_{n;t}^{b_n}
=\frac{A_{1;t^\prime}^{-b_1}A_{2;t^\prime}^{b_2}\cdots A_{n;t^\prime}^{b_n}}{P^{-b_1}},\nonumber
\end{eqnarray}
which is the expansion of $N$ with respect to $A_{t^\prime}$. By $N\in\mathcal U$, we get that $$N=\frac{A_{1;t^\prime}^{-b_1}A_{2;t^\prime}^{b_2}\cdots A_{n;t^\prime}^{b_n}}{P^{-b_1}}\in\mathcal L(t^\prime)=\mathbb {KP}[A_{1;t^\prime}^{\pm1},\ldots,A_{n;t^\prime}^{\pm1}].$$ Thus we have
$1/P^{-b_1}\in\mathcal L(t^\prime).$
Because both $P^{-b_1}$ and $1/P^{-b_1}$ are in $\mathcal L(t^\prime)$, we must have $P^{-b_1}=\lambda A_{1;t^\prime}^{c_1}\cdots A_{n;t^\prime}^{c_n}$ for some invertible element $\lambda$ in $\mathbb{KP}$ and $c_1,\ldots,c_n\in\mathbb Z$. Since $P$ is a polynomial in $\mathbb{KP}[A_{2;t^\prime},\ldots,A_{n;t^\prime}]$ with $A_{i;t^\prime}\nmid P$ for any $i$, we get that $c_1=\ldots=c_n=0$ and $P^{-b_1}=\lambda$ is invertible in $\mathbb{KP}$.  By our assumption in Subsection \ref{rmkexbi}, we know that  $P^{-b_1}=\lambda$ is not invertible in $\mathbb{KP}$. This concludes a contradiction. So we must have $b_1,\ldots,b_n\geq 0$.
\end{proof}

\begin{proposition}\label{promono}
Let $\mathcal U$ be an upper cluster algebra with coefficient semifield $\mathbb P$  and $t_0, t$ two seeds of $\mathcal U$.
 Let $M=A_{1;t_0}^{c_1}\cdots A_{n;t_0}^{c_n}$ be a cluster monomial in $A_{t_0}$ and $N=A_{1;t}^{b_1}\cdots A_{n;t}^{b_n}$ a Laurent monomial in $A_{t}$.
 Let $$I=\{i\mid c_i>0\}\;\;\text{and}\;\;\;J=\{j\mid b_j>0\}.$$ If there exists an invertible element $\lambda$ in $\mathbb {KP}$ such that $M=\lambda N$, then
 \begin{itemize}
 \item[(i)] $N$ is a cluster monomial in $A_{t}$, i.e., we have $b_1,\ldots,b_n\geq0$;

 \item[(ii)] there exists a bijection $\sigma:I\rightarrow J$ such that $A_{i;t_0}=A_{\sigma(i);t}$ and $c_{i}=b_{\sigma(i)}$. In particular, we have $M=N$.
 \end{itemize}
\end{proposition}
\begin{proof}(i)
By $M=\lambda N$, we know that $N=\lambda^{-1}M\in\mathcal U$. Then by Lemma \ref{lemmonomial}, we know that $N$ is a cluster monomial in $A_{t}$.

(ii) Let $A_{j;t}$ be a cluster variable appearing in $N$. Then by (i), we have $b_j>0$ and $j\in J$. Because 
\begin{eqnarray}
\sum\limits_{i\in I}^nc_i(A_{j;t}\mid\mid A_{i;t_0})_d=(A_{j;t}\mid\mid M)_d=(A_{j;t}\mid\mid \lambda N)_d=(A_{j;t}\mid\mid  N)_d=-b_j<0,\nonumber
\end{eqnarray}
 there must exist some $j^\prime\in I$  such that $$(A_{j;t}\mid\mid A_{j^\prime;t_0})_d<0.$$ Then by Proposition \ref{rmkdf} (4), we have $A_{j;t}=A_{j^\prime;t_0}$. So we have
 $$c_{j^\prime}=-(A_{j^\prime;t_0}\mid\mid M)_d=-(A_{j;t}\mid\mid M)_d=-(A_{j;t}\mid\mid N)_d=b_j.$$

Similarly, we can show that any cluster variable $A_{i;t_0}$ appearing in $M$ also appears in $N$, and the multiplicity of $A_{i;t_0}$ in $N$ is equal to the multiplicity of $A_{i;t_0}$ in $M$. Then the result follows.
\end{proof}

\begin{theorem}[Local factorization of cluster monomials]\label{thmup}
Let $\mathcal U$ be an upper cluster algebra with coefficient semifield $\mathbb P$  and  initial seed $t_0$.  Then the following statements hold.

\begin{itemize}

\item[(i)]  If $L$ is a cluster monomial of $\mathcal U$ which does not contain the initial cluster variable $A_{k;t_0}$ for some $k$, then $L/A_{k;t_0}\notin\mathcal U$, that is, $(A_{k;t_0}\mid\mid L)_v=0$.

\item[(ii)] If $L$ is a cluster monomial in non-initial cluster variables, then  $L/A_{k;t_0}\notin\mathcal U$, that is, $(A_{k;t_0}\mid\mid L)_v=0$ for $k=1,\ldots,n$.
    \item[(iii)] Let $M=\prod\limits_{i=1}^nA_{i;t}^{b_i}$ be a cluster monomial in a seed $t$ of $\mathcal U$ and $I=\{i\mid A_{i;t}\in A_{t_0}\}$, then $M=\prod\limits_{i\in I}A_{i;t}^{b_i}\cdot\prod\limits_{j\notin I}A_{j;t}^{b_j}$ is a local factorization of $M$ with respect to $t_0$.
\end{itemize}
\end{theorem}

\begin{proof}
(i) Assume by contradiction that  $L/A_{k;t_0}\in\mathcal U$. In the following proof, we will deduce a contradiction.

Because the initial cluster variable $A_{k;t_0}$ does not appear in $L$, there exists a seed $t$ such that $L$ is a cluster monomial in $A_t$ and $A_{k;t_0}$ does not belong to $A_t$. Consider the Laurent expansion of $A_{k;t_0}$ with respect to $A_t$, say
$$A_{k;t_0}=\frac{P(A_{1;t},\ldots,A_{n;t})}{A_{1;t}^{d_1^\prime}\cdots A_{n;t}^{d_n^\prime}},$$
where $P\in\mathbb {KP}[A_{1;t},\ldots,A_{n;t}]$ with $A_{i;t}\nmid P$ for $i=1,\ldots,n$.

Now we have that
\begin{eqnarray}
L/A_{k;t_0}=\frac{L\cdot (A_{1;t}^{d_1^\prime}\cdots A_{n;t}^{d_n^\prime})}{P},\nonumber
\end{eqnarray}
which is the expansion of $L/A_{k;t_0}$ with respect to $A_t$.
By $L/A_{k;t_0}\in\mathcal U$, we have  $$\frac{L\cdot (A_{1;t}^{d_1^\prime}\cdots A_{n;t}^{d_n^\prime})}{P} \in{\mathcal L}(t):=\mathbb {KP}[A_{1;t}^{\pm 1},\ldots,A_{n;t}^{\pm1}].$$
 Because $L\cdot (A_{1;t}^{d_1^\prime}\cdots A_{n;t}^{d_n^\prime})$ is a monic Laurent monomial in  ${\mathcal L}(t)$, we get
 $1/P\in{\mathcal L}(t)$. Since both $P$ and $1/P$ are in ${\mathcal L}(t)$, we must have $P=\lambda A_{1;t}^{c_1}\cdots A_{n;t}^{c_n}$ for some invertible element $\lambda$ in $\mathbb{KP}$ and $c_1,\ldots,c_n\in\mathbb Z$. Because  $P$ is a polynomial in $\mathbb {KP}[A_{1;t},\ldots,A_{n;t}]$ with $A_{i;t}\nmid P$ for any $i$, we get $c_1=\ldots=c_n=0$ and thus $P=\lambda$ is invertible in $\mathbb{KP}$.
  Applying  Proposition \ref{promono} to
  $$A_{k;t_0}=\frac{P}{A_{1;t}^{d_1^\prime}\cdots A_{n;t}^{d_n^\prime}}=\lambda\cdot A_{1;t}^{-d_1^\prime}\cdots A_{n;t}^{-d_n^\prime},$$
  we get  $A_{k;t_0}\in A_t$, which contradicts $A_{k;t_0}\notin A_t$. So $L/A_{k;t_0}\notin\mathcal U$.

  (ii) This follows from (i).

  (iii)  This follows from (ii) and the definition of local factorization.
\end{proof}

\begin{corollary}\label{cordvec}
 Let $\mathcal U$ be a full rank upper cluster algebra with initial seed $t_0$. Let
 $$M=\frac{P_M(A_{1;t_0},\ldots,A_{n;t_0})}{A_{1;t_0}^{d_1}\cdots A_{n;t_0}^{d_n}}$$ be a cluster monomial of $\mathcal U$,  where ${\bf d}_M:=(d_1,\ldots,d_n)^{\rm T}$ is the $d$-vector of $M$ with respect to $A_{t_0}$. Then the following statements hold.
 \begin{itemize}
 \item[(i)]  If $M$ does not contain the initial cluster variable $A_{k;t_0}$ for some $k$, then we have
 $$d_k=(A_{k;t_0}\mid\mid M)_d=(A_{k;t_0}\mid\mid P_M)_v\geq 0.$$

 \item[(ii)] If $M$ is a cluster monomial in non-initial cluster variables, then ${\bf d}_M$ and $M$ are uniquely determined by the numerator polynomial $P_M$.
 \end{itemize}
\end{corollary}
\begin{proof}
(i) By Theorem \ref{thmup} (i), we know that $(A_{k;t_0}\mid\mid M)_v=0$. Then by Theorem \ref{thmind} (ii), we get $d_k=(A_{k;t_0}\mid\mid M)_d=(A_{k;t_0}\mid\mid P_M)_v\geq 0$.

(ii) This follows from (i).
\end{proof}

Actually, the above result can be extended from  cluster monomials to monomials in cluster variables if we further assume that $\mathcal U$ is primitive or factorial.

\begin{proposition}\label{prodvec}
 Let $\mathcal U$ be a full rank, primitive upper cluster algebra with initial seed $t_0$. Let
 $$M=\frac{P_M(A_{1;t_0},\ldots,A_{n;t_0})}{A_{1;t_0}^{d_1}\cdots A_{n;t_0}^{d_n}}$$ be a  monomial in cluster variables of $\mathcal U$ (not necessarily a cluster monomial),  where ${\bf d}_M:=(d_1,\ldots,d_n)^{\rm T}$ is the $d$-vector of $M$ with respect to $A_{t_0}$. Then the following statements hold.
 \begin{itemize}
 \item[(i)]  If $M$ does not contain the initial cluster variable $A_{k;t_0}$ for some $k$, then we have $(A_{k;t_0}\mid\mid M)_v=0$ and $d_k=(A_{k;t_0}\mid\mid M)_d=(A_{k;t_0}\mid\mid P_M)_v\geq 0.$

 \item[(ii)] If $M$ is a monomial in non-initial cluster variables, then ${\bf d}_M$ and $M$ are uniquely determined by the numerator polynomial $P_M$.
 \end{itemize}

\end{proposition}
 \begin{proof}
(i) Since $M$ is a monomial in cluster variables and it does not contain  $A_{k;t_0}$, we can assume that $M=\prod\limits_{i=1}^sM_i$, where each $M_i$ is a cluster variable of $\mathcal U$ different from $A_{k;t_0}$. By Theorem \ref{thmup} (i), we have $(A_{k;t_0}\mid\mid M_i)_v=0$, for $i=1,\ldots,s$.

Since $\mathcal U$ is a full rank, primitive upper cluster algebra, we know that $\mathcal U$ is factorial, by Theorem \ref{thmufd}. Then by Theorem \ref{thmcufd} (i)(iii),  we know that
$$(A_{k;t_0}\mid\mid M)_v=\sum\limits_{i=1}^s(A_{k;t_0}\mid\mid M_i)_v=0.$$
Then by Theorem \ref{thmind} (ii), we get $d_k=(A_{k;t_0}\mid\mid M)_d=(A_{k;t_0}\mid\mid P_M)_v\geq 0$.

(ii) This follows from (i).
 \end{proof}

\section{Application to \texorpdfstring{$F$}{Lg}-polynomials}\label{sec6}
In this section, we always assume that $\mathcal U$ is an upper cluster algebra with principal coefficients at $t_0$.
\subsection{From \texorpdfstring{$F$}{Lg}-polynomials to monomials in non-initial cluster variables}\label{sec61}

In this subsection, we prove that if $M$ is a monomial in non-initial cluster variables, then $M$ is uniquely determined by its $F$-polynomial.

Let $M$ be a monomial in cluster variables of $\mathcal U$, say $M=\prod\limits_{i=1}^sM_i$, where each $M_i$ is a cluster variable of $\mathcal U$. The  {\em $F$-polynomial} $F_M$ of $M$ is defined to be the polynomial
$F_M:=\prod\limits_{i=1}^sF_{M_i},$
where $F_{M_i}$ is the $F$-polynomial of the cluster variable $M_i$.
The {\em $g$-vector}  of $M$ is defined to be the vector ${\bf g}_M:=\sum\limits_{i=1}^s{\bf g}_{M_i}$, where ${\bf g}_{M_i}$ is the $g$-vector of $M_i$.

\begin{theorem}\label{thmfpoly}
 Let $\mathcal U$ be an upper cluster algebra with principal coefficients at $t_0$. Let $M, N$ be two  monomials in non-initial cluster variables of $\mathcal U$ and $F_M, F_N$ their $F$-polynomials.   If $F_M=F_N$,  then $M=N$. In particular, ${\bf g}_M={\bf g}_N$, where ${\bf g}_M$ and ${\bf g}_N$ are $g$-vectors of $M$ and $N$.
\end{theorem}
\begin{proof}
Let ${\bf d}_M$ and ${\bf d}_N$ be the $d$-vectors of $M$ and $N$ with respect to $A_{t_0}$. We know that $M$ and $N$ have the form
\begin{eqnarray}\label{eqnpd}
M=P_M/A_{t_0}^{{\bf d}_M}\;\;\text{ and }\;\;\;N=P_N/A_{t_0}^{{\bf d}_N},
\end{eqnarray}
where $P_M,P_N\in\mathbb Z[Z_1,\ldots,Z_n;A_{1;t_0},\ldots,A_{n;t_0}]$ with $A_{i;t_0}\nmid P_M$ and $A_{i;t_0}\nmid P_N$ for any $i=1,\ldots,n$.

By Theorem \ref{thmsf} (i), we know that
\begin{eqnarray}
M&=&A_{t_0}^{{\bf g}_M}F_M(\widehat X_{1;t_0},\ldots,\widehat X_{n;t_0}),\nonumber\\
N&=&A_{t_0}^{{\bf g}_N}F_N(\widehat X_{1;t_0},\ldots,\widehat X_{n;t_0}),\nonumber
\end{eqnarray}
where $\widehat X_{j;t_0}=Z_j\prod\limits_{i=1}^nA_{i;t_0}^{b_{ij}^{t_0}}$.
So we can get the following equalities:
\begin{eqnarray}
F_M(\widehat X_{1;t_0},\ldots,\widehat X_{n;t_0})&=&P_M/A_{t_0}^{{\bf d}_M+{\bf g}_M},\nonumber\\
F_N(\widehat X_{1;t_0},\ldots,\widehat X_{n;t_0})&=&P_N/A_{t_0}^{{\bf d}_N+{\bf g}_N}.\nonumber
\end{eqnarray}
Because $P_M$ and $P_N$ are not divisible by any $A_{i;t_0}$, we know that ${\bf d}_M+{\bf g}_M$ and ${\bf d}_N+{\bf g}_N$ are the $d$-vectors of $F_M(\widehat X_{1;t_0},\ldots,\widehat X_{n;t_0})$ and $F_N(\widehat X_{1;t_0},\ldots,\widehat X_{n;t_0})$ with respect to $t_0$. By $F_M=F_N$, we have $$F_M(\widehat X_{1;t_0},\ldots,\widehat X_{n;t_0})=F_N(\widehat X_{1;t_0},\ldots,\widehat X_{n;t_0}),$$ which implies that ${\bf d}_M+{\bf g}_M={\bf d}_N+{\bf g}_N$ and $P_M=P_N$.

By $P_M=P_N$ and applying Proposition \ref{prodvec} to $M$ and $N$, we get ${\bf d}_M={\bf d}_N$. Then by the equality (\ref{eqnpd}), we know that $M=N$. In particular, we have ${\bf g}_M={\bf g}_N$. This completes the proof.
\end{proof}

\subsection{\texorpdfstring{$F$}{Lg}-polynomials of non-initial cluster variables are irreducible}\label{sec62}
 In this subsection, we prove that the $F$-polynomials of non-initial cluster variables are  irreducible in $\mathbb K[Z_1,\ldots,Z_n]$.

\begin{theorem}\label{thmffirr}
Let $\mathcal U$ be an upper cluster algebra with principal coefficients at $t_0$ and $A_{k;t}$ a non-initial cluster variable of $\mathcal U$.  Then the $F$-polynomial $F_{k;t}$ of $A_{k;t}$ is irreducible in $\mathbb K[Z_1,\ldots,Z_n]$.
\end{theorem}
\begin{proof} We claim that $F_{k;t}$ is not a constant. Otherwise,  the $f$-compatibility degree $(A_{i;t_0}\mid\mid A_{k;t})_f=0$ for $i=1,\ldots,n$. Then by Proposition \ref{rmkdf} (2), we get that $A_{k;t}$ is compatible with any cluster variable in $A_{t_0}$. Thus $A_{k;t}\in A_{t_0}$, by Theorem \ref{thmcomp}. This contradicts that $A_{k;t}$ is a non-initial cluster variable of $\mathcal U$. So $F_{k;t}$ is not a constant.

Since $F_{k;t}$ is not a constant, it can not be  invertible in  $\mathbb K[Z_1,\ldots,Z_n]$. So we can discuss its irreducibility. Assume there exist two polynomials $F_1,F_2\in \mathbb K[Z_1,\ldots,Z_n]$ such that
\begin{eqnarray}\label{eqnf}
F_{k;t}(Z_1,\ldots,Z_n)=F_1(Z_1,\ldots,Z_n)F_2(Z_1,\ldots,Z_n).
\end{eqnarray}
We will show that either $F_1\in\mathbb K^\times$ or $F_2\in\mathbb K^\times$.

For $1\leq j\leq n$, put $\widehat X_j=Z_j\prod\limits_{i=1}^nA_{i;t_0}^{b_{ij}^{t_0}}$.
Consider the Laurent polynomials
$F_1(\widehat X_1,\ldots,\widehat X_n)$ and $F_2(\widehat X_1,\ldots,\widehat X_n)$
in $\mathbb K[Z_1,\ldots,Z_n][A_{1;t_0}^{\pm1},\ldots,A_{n;t_0}^{\pm1}]$.
We choose two vectors ${\bf d}_1, {\bf d}_2$ in $\mathbb N^n$ such that
\[
M_1:=A_{t_0}^{{\bf d}_1}F_1(\widehat X_1,\ldots,\widehat X_n)\;\;\;\; \text{and}\;\;\; \;
M_2:=A_{t_0}^{{\bf d}_2}F_2(\widehat X_1,\ldots,\widehat X_n)
\]
are polynomials in
$\mathbb K[Z_1,\ldots,Z_n][A_{1;t_0},\ldots,A_{n;t_0}]\subseteq\mathcal U$. By the local unique factorization property, there exist unique
cluster monomials $N_1=A_{t_0}^{{\bf v}_1}, N_2=A_{t_0}^{{\bf v}_2}$ such that
$$L_1:=M_1/N_1=A_{t_0}^{{\bf d}_1-{\bf v}_1}F_1(\widehat X_1,\ldots,\widehat X_n)$$ and
$$L_2:=M_2/N_2=A_{t_0}^{{\bf d}_2-{\bf v}_2}F_2(\widehat X_1,\ldots,\widehat X_n)$$
are in $\mathcal U$ and satisfy $$(A_{j;t_0}\mid\mid L_1)_v=0=(A_{j;t_0}\mid\mid L_2)_v$$ for $j=1,\ldots,n$. Then by Corollary \ref{corzp}, we have $$(A_{j;t_0}\mid\mid L_1L_2)_v=(A_{j;t_0}\mid\mid L_1)_v+(A_{j;t_0}\mid\mid L_2)_v=0$$ for $j=1,\ldots,n$.

By Theorem \ref{thmsf} (i), we know that the cluster variable $A_{k;t}$ has the following form:
 \begin{eqnarray}
A_{k;t}=A_{t_0}^{{\bf g}}\cdot F_{k;t}(\widehat X_1,\ldots,\widehat X_n),\nonumber
\end{eqnarray}
where ${\bf g}\in \mathbb Z^n$.
By the equality (\ref{eqnf}), we have that
$$F_{k;t}(\widehat X_1,\ldots,\widehat X_n)=F_1(\widehat X_1,\ldots,\widehat X_n)F_2(\widehat X_1,\ldots,\widehat X_n).$$
So we know that $$A_{k;t}/(L_1L_2)=A_{t_0}^{{\bf g}+{\bf v}_1+{\bf v}_2-{\bf d}_1-{\bf d}_2}.$$
Thus we have
$A_{k;t}=A_{t_0}^{\bf v}(L_1L_2)$, where ${\bf v}={\bf g}+{\bf v}_1+{\bf v}_2-{\bf d}_1-{\bf d}_2$.

Now we show that ${\bf v}=0$ and $A_{k;t}=L_1L_2$. We choose
two vectors ${\bf v}^+$ and ${\bf v}^-$ in $\mathbb N^n$ such that
${\bf v}={\bf v}^+-{\bf v}^-$. Then by
$$
A_{k;t}=A_{t_0}^{\bf v}(L_1L_2)=A_{t_0}^{{\bf v}^+-{\bf v}^-}(L_1L_2),
$$
we get that
\begin{eqnarray}\label{eqnfac}
M:=A_{t_0}^{{\bf v}^-}A_{k;t}=A_{t_0}^{{\bf v}^+}L_1L_2.
\end{eqnarray}
Since $F_{k;t}$ is not a constant, we know that $A_{k;t}$ is not an initial cluster variable. Then by  Theorem \ref{thmup},
 we have $(A_{j;t_0}\mid\mid A_{k;t})_v=0$ for $j=1,\ldots,n$. On the other hand, we know that $L_1L_2$ also satisfies   $(A_{j;t_0}\mid\mid L_1L_2)_v=0$ for $j=1,\ldots,n$. So the equality (\ref{eqnfac}) gives two local factorizations of $M$ with respect to $t_0$. Then by the uniqueness in Theorem \ref{thmunique}, we get that $${\bf v}^+={\bf v}^-\;\;\text{and}\;\;\;A_{k;t}=L_1L_2.$$

By Theorem \ref{thmgls} (ii), we know that $\mathcal U^\times=\{\lambda Z_1^{c_1}\cdots Z_n^{c_n}\mid \lambda\in\mathbb K^\times,\;c_1,\ldots,c_n\in \mathbb Z\}$ and that $A_{k;t}$ is irreducible in $\mathcal U$. So the factorization $A_{k;t}=L_1L_2$ implies either $L_1=\lambda Z_1^{c_1}\cdots Z_n^{c_n}$ or $L_2=\lambda Z_1^{c_1}\cdots Z_n^{c_n}$ for some $\lambda\in\mathbb K^\times$ and $(c_1,\ldots,c_n)\in\mathbb Z^n$.
So either $$F_1(Z_1,\ldots,Z_n)=L_1\mid _{A_{1;t_0}=\ldots=A_{n;t_0}=1}=\lambda Z_1^{c_1}\cdots Z_n^{c_n}$$ or $$F_2(Z_1,\ldots,Z_n)=L_2\mid _{A_{1;t_0}=\ldots=A_{n;t_0}=1}=\lambda Z_1^{c_1}\cdots Z_n^{c_n}$$ for some $\lambda\in\mathbb K^\times$ and $(c_1,\ldots,c_n)\in\mathbb Z^n$. Because both $F_1$ and $F_2$ are polynomials in $\mathbb K[Z_1,\ldots,Z_n]$, we must have $c_1,\ldots,c_n\geq 0$. Now by Proposition \ref{profy}, we know that $F_{k;t}=F_1F_2$ is not divisible by any $Z_j$. So we must have $c_1=\ldots=c_n=0$. Thus either $F_1=\lambda$ or $F_2=\lambda$ for some $\lambda\in\mathbb K^\times$. So $F_{k;t}$ is irreducible in $\mathbb K[Z_1,\ldots,Z_n]$.
\end{proof}
\begin{remark}
Garcia Elsener {\em et al.} in \cite[Theorem 3.9]{GLS19} prove that the $F$-polynomials of non-initial cluster variables are irreducible for {\em factorial principal coefficient cluster algebras}. In fact, their proof still works for factorial principal coefficient {\em upper} cluster algebras. Note that the factoriality of principal coefficient upper cluster algebras is no longer a problem, thanks to Theorem \ref{thmufd}. Thus one can also use the method in \cite{GLS19} to show the irreducibility of $F$-polynomials of non-initial cluster variables.
\end{remark}

\section{Application to combinatorics of cluster Poisson variables}\label{sec7}

\subsection{Parametrization of \texorpdfstring{$\mathscr A$}{Lg}-exchange pairs by cluster Poisson variables}\label{sec71}
 In this subsection, we give several equivalent characterizations of when two cluster Poisson variables are equal. As an application, we prove that the $\mathscr A$-exchange pairs are parameterized by the cluster Poisson variables.

\begin{proposition}\label{protwo}
Let $\mathcal U$ be an upper  cluster algebra with coefficient semifield $\mathbb P$ and  initial seed $t_0$.
Let $M_1, M_2$ be two cluster monomials of $\mathcal U$ and $N_1, N_2$  two different cluster monomials in $t_0$. If $$p_1M_1+p_2M_2=p_1^\prime N_1+p_2^\prime N_2$$ holds for some $p_1,p_2,p_1^\prime,p_2^\prime\in\mathbb{P}$, then either
$(M_1, M_2, p_1, p_2)=(N_1, N_2, p_1^\prime, p_2^\prime)$
or
$(M_1, M_2, p_1, p_2)=(N_2, N_1, p_2^\prime, p_1^\prime)$ holds.
\end{proposition}
\begin{proof}
Since both $N_1$ and $N_2$ are cluster monomials in $A_{t_0}$, we have
 $$(A_{i;t_0}\mid\mid p_1^\prime N_1)_d=(A_{i;t_0}\mid\mid N_1)_d\leq 0$$
 and
  $$(A_{i;t_0}\mid\mid  p_2^\prime N_2)_d=(A_{i;t_0}\mid\mid  N_2)_d\leq 0$$
for $i=1,\ldots,n$.
By the positivity of the Laurent phenomenon and since the coefficients in $p_1M_1+p_2M_2$ and $p_1^\prime N_1+p_2^\prime N_2$ are positive, we have
\begin{eqnarray}
&&{\rm max}\{(A_{i;t_0}\mid\mid   M_1)_d,(A_{i;t_0}\mid\mid  M_2)_d\}\nonumber\\
&=&{\rm max}\{(A_{i;t_0}\mid\mid  p_1 M_1)_d,(A_{i;t_0}\mid\mid  p_2M_2)_d\}\nonumber\\
&=&(A_{i;t_0}\mid\mid   p_1M_1+p_2M_2)_d=(A_{i;t_0}\mid\mid   p_1^\prime N_1+p_2^\prime N_2)_d\nonumber\\
&=&{\rm max}\{(A_{i;t_0}\mid\mid  p_1^\prime N_1)_d,(A_{i;t_0}\mid\mid  p_2^\prime N_2)_d\}\leq 0,\nonumber
\end{eqnarray}
for  $i=1,\ldots,n$.

We claim that $M_1$ and $M_2$ are cluster monomials in $t_0$. Let $A_{k;t}$ be a cluster variable appearing in $M_1$ or $M_2$. By Corollary \ref{cordmonomial} (ii) and $${\rm max}\{(A_{i;t_0}\mid\mid   M_1)_d,(A_{i;t_0}\mid\mid  M_2)_d\}\leq0,$$ we know that
$A_{k;t}$ is compatible with $A_{i;t_0}$, where $i=1,\ldots,n$. Then by Theorem \ref{thmcomp}, we  get $A_{k;t}\in A_{t_0}$. So $M_1$ and $M_2$ are cluster monomials in $t_0$.

 Hence,  the polynomial  $p_1M_1+p_2M_2$ is actually a polynomial in variables from $A_{t_0}$.  Since $N_1$ and $N_2$ are two different cluster monomials in $t_0$, we must have $M_1\neq M_2$. Otherwise, $p_1^\prime N_1+p_2^\prime N_2$ can not be equal to $p_1M_1+p_2M_2$. Now the result follows from  the algebraic independence of $A_{1;t_0},\ldots,A_{n;t_0}$.
\end{proof}

\begin{proposition}\label{lemdex}
Let $\mathcal U$ be an upper cluster algebra with coefficient semifield $\mathbb P$ and initial seed $t_0$.
Let  $A_{k;t}$ be a cluster variable of $\mathcal U$. Then the following two statements are equivalent.
\begin{itemize}
\item[(i)] $(A_{j;t_0},A_{k;t})$ is an $\mathscr A$-exchange pair associated with the mutation $t_1:=\mu_j(t_0)$.

\item[(ii)] $(A_{j;t_0}\mid\mid  A_{k;t})_d>0$ and $(A_{i;t_0}\mid\mid  A_{k;t})_d=0$ for any $i\neq j$.
\end{itemize}
\end{proposition}
\begin{proof}
(i)$\Longrightarrow$ (ii): This is clear.

(ii)$\Longrightarrow$ (i): By Proposition \ref{rmkdf} (4) and the assumption, we get that $A_{k;t}\notin A_{t_0}$ and  $\{A_{k;t}\}\cup(A_{t_0}\backslash\{A_{j;t_0}\})$ is a compatible set with $n$ elements.
Then by Theorem \ref{thmcomp}, we know that $\{A_{k;t}\}\cup(A_{t_0}\backslash\{A_{j;t_0}\})$ is a cluster of $\mathcal U$, say $A_{t_2}=\{A_{k;t}\}\cup(A _{t_0}\backslash\{A_{j;t_0}\})$. Since the seeds $t_2$ and $t_0$ have exactly $n-1$ common cluster variables,  there is an edge between $t_2$ and $t_0$ in the exchange graph of $\mathcal U$, by Corollary \ref{corn-1}. Then the result follows.
\end{proof}

\begin{proposition}\label{proexb}
Let  $\mathcal U$ and  $\mathcal U^\prime$ be any two upper cluster algebras with the same initial exchange matrix $B_{t_0}$ at $t_0$.  For any vertex $t$ of $\mathbb T_n$, denote by $A_{t}$ and $A_{t}^\prime$  the  cluster  of $\mathcal U$ and  $\mathcal U^\prime$ at the vertex $t$. Then we have
\begin{itemize}
\item[(i)] $A_{i;t_1}=A_{j;t_2}$ in $\mathcal U$ if and only if $A_{i;t_1}^\prime=A_{j;t_2}^\prime$ in $\mathcal U^\prime$;

\item[(ii)] The map $\rho: A_{k;t}\mapsto A_{k;t}^{\prime}$ gives a bijection from  cluster variables of $\mathcal U$ to cluster variables of $\mathcal U^\prime$;

\item[(iii)] $\rho$ induces a bijection from the  set of  $\mathscr A$-exchange pairs of $\mathcal U$ to that of $\mathcal U^\prime$.
\end{itemize}
\end{proposition}
\begin{proof}
(i) Because $\mathcal U$ and $\mathcal U^\prime$ have the same initial exchange matrix $B_{t_0}$ at $t_0$, we know that they have the same exchange matrix at each vertex $t$ of $\mathbb T_n$.
 Notice that both $(A_{i;t_1}\mid\mid  A_{j;t_2})_d$ and $(A_{i;t_1}^\prime\mid\mid  A_{j;t_2}^\prime)_d$ are defined by the $(i,j)$-entry of the $D$-matrix $D_{t_2}^{B_{t_1}; t_1}$, which is independent of the choice of coefficient semifield. So
  we have $(A_{i;t_1}\mid\mid  A_{j;t_2})_d=(A_{i;t_1}^\prime\mid\mid  A_{j;t_2}^\prime)_d$.  By Proposition \ref{rmkdf} (4), we know that $A_{i;t_1}=A_{j;t_2}$ in $\mathcal U$  if and only if $(A_{i;t_1}\mid\mid  A_{j;t_2})_d=-1$ if and only if $(A_{i;t_1}^\prime\mid\mid  A_{j;t_2}^\prime)_d=-1$ if and only if $A_{i;t_1}^\prime=A_{j;t_2}^\prime$ in $\mathcal U^\prime$.

(ii) This follows from (i).

 (iii) This follows from (ii) and the characterization of  $\mathscr A$-exchange pairs in Proposition \ref{lemdex}.
\end{proof}

\begin{lemma} \label{lemb}

Let $\mathcal U_{\rm uc}$ be an upper cluster algebra with universal coefficient semifield and $t, t_0$ two seeds of  $\mathcal U_{\rm uc}$. Let $t^\prime=\mu_k(t)$ and $t_1=\mu_j(t_0)$. Let $\frac{X_{k;t}}{1+X_{k;t}}M_1+\frac{1}{1+X_{k;t}}M_2$ be the $k$-th exchange binomial of $t$ and $\frac{X_{j;t_0}}{1+X_{j;t_0}}N_1+\frac{1}{1+X_{j;t_0}}N_2$  the $j$-th exchange binomial of $t_0$.
If $$\frac{X_{k;t}}{1+X_{k;t}}M_1+\frac{1}{1+X_{k;t}}M_2=\frac{X_{j;t_0}}{1+X_{j;t_0}}N_1+\frac{1}{1+X_{j;t_0}}N_2,$$
then either $$(A_{k;t},A_{k;t^\prime})=(A_{j;t_0},A_{j;t_1})\;\;\text{ or }\;\;\;(A_{k;t},A_{k;t^\prime})=(A_{j;t_1},A_{j;t_0}).$$
\end{lemma}
\begin{proof}
By the assumption, we know that $$A_{k;t}A_{k;t^\prime}=\frac{X_{k;t}}{1+X_{k;t}}M_1+\frac{1}{1+X_{k;t}}M_2=\frac{X_{j;t_0}}{1+X_{j;t_0}}N_1+\frac{1}{1+X_{j;t_0}}N_2.$$
The above equality can be viewed as the Laurent expansion of $A_{k;t}A_{k;t^\prime}$ with respect to $A_{t_0}$. Clearly, this Laurent expansion is a polynomial in $A_{t_0}$.

In the following proof, we distinguish two cases. Case (a): $(A_{i;t_0}\mid\mid  A_{k;t})_d\leq 0$ for any $i=1,\ldots,n$. Case (b): there exists some $i_0$ such that $(A_{i_0;t_0}\mid\mid  A_{k;t})_d>0$.

Case (a): If $(A_{i;t_0}\mid\mid  A_{k;t})_d\leq 0$ for any $i=1,\ldots,n$, then $A_{k;t}$ is compatible with any cluster variable in $A_{t_0}$. By Theorem \ref{thmcomp}, we know that $A_{t_0}$ is a maximal compatible set. So $A_{k;t}\in A_{t_0}$, say $A_{k;t}=A_{i_0;t_0}$. We will show that $$i_0=j\;\;\text{and}\;\;\;A_{i_0;t_0}=A_{j;t_0}.$$

 We know that
 $$A_{k;t^\prime}=A_{k;t}^{-1}\cdot(\frac{X_{j;t_0}}{1+X_{j;t_0}}N_1+\frac{1}{1+X_{j;t_0}}N_2)
 =A_{i_0;t_0}^{-1}\cdot(\frac{X_{j;t_0}}{1+X_{j;t_0}}N_1+\frac{1}{1+X_{j;t_0}}N_2),$$
 which is the Laurent expansion of $A_{k;t^\prime}$ with respect to $A_{t_0}$. So we have $$(A_{i_0;t_0}\mid\mid  A_{k;t^\prime})_d=1>0\;\;\text{ and }\;\;\;(A_{i;t_0}\mid\mid  A_{k;t^\prime})_d=0$$ for any $i\neq i_0$.
 Then by Proposition \ref{lemdex}, we get that $$A_{k;t^\prime}=A_{i_0;t_2},$$ where $t_2=\mu_{i_0}(t_0)$.

 Let $\frac{X_{i_0;t_0}}{1+X_{i_0;t_0}}L_1+\frac{1}{1+X_{i_0;t_0}}L_2$ be the $i_0$-th exchange binomial of $t_0$. By Proposition \ref{protwo} and the following equality,
 $$\frac{X_{j;t_0}}{1+X_{j;t_0}}N_1+\frac{1}{1+X_{j;t_0}}N_2=A_{k;t}A_{k;t^\prime}=A_{i_0;t_0}A_{i_0;t_2}
=\frac{X_{i_0;t_0}}{1+X_{i_0;t_0}}L_1+\frac{1}{1+X_{i_0;t_0}}L_2,$$
we get that either $X_{j;t_0}=X_{i_0;t_0}$ or $X_{j;t_0}=X_{i_0;t_0}^{-1}$. Since $X_{1;t_0},\ldots,X_{n;t_0}$ are algebraically independent, we must have $X_{j;t_0}=X_{i_0;t_0}$ and $j=i_0$. So $$A_{k;t}=A_{i_0;t_0}=A_{j;t_0}.$$ Then by the following equality,
$$A_{k;t}A_{k;t^\prime}=\frac{X_{k;t}}{1+X_{k;t}}M_1+\frac{1}{1+X_{k;t}}M_2=\frac{X_{j;t_0}}{1+X_{j;t_0}}N_1+\frac{1}{1+X_{j;t_0}}N_2
=A_{j;t_0}A_{j;t_1},$$
we  get $A_{k;t^\prime}=A_{j;t_1}$ and thus $(A_{k;t},A_{k;t^\prime})=(A_{j;t_0},A_{j;t_1})$.

Case (b): There exists some $i_0$ such that $(A_{i_0;t_0}\mid\mid  A_{k;t})_d>0$. By the fact that the expansion of $A_{k;t}A_{k;t^\prime}$ with respect to $A_{t_0}$ is a polynomial, we know that $$0\geq(A_{i_0;t_0}\mid\mid  A_{k;t}A_{k;t^\prime})_d=(A_{i_0;t_0}\mid\mid  A_{k;t})_d+ (A_{i_0;t_0}\mid\mid  A_{k;t^\prime})_d.$$
So we must have $(A_{i_0;t_0}\mid\mid  A_{k;t^\prime})_d<0$. Then by Proposition \ref{rmkdf} (4), we get $$A_{k;t^\prime}=A_{i_0;t_0}.$$  Similarly to the proof in Case (a), we can prove that $i_0=j$ and thus $$A_{k;t^\prime}=A_{i_0;t_0}=A_{j;t_0}.$$ Then $A_{k;t}=A_{j;t_1}$ follows from the following equality
$$A_{k;t}A_{k;t^\prime}=\frac{X_{k;t}}{1+X_{k;t}}M_1+\frac{1}{1+X_{k;t}}M_2=\frac{X_{j;t_0}}{1+X_{j;t_0}}N_1+\frac{1}{1+X_{j;t_0}}N_2
=A_{j;t_0}A_{j;t_1}.$$ So we have $(A_{k;t},A_{k;t^\prime})=(A_{j;t_1},A_{j;t_0})$. This completes the proof.
\end{proof}

\begin{theorem}\label{mainthm}
Let $\mathcal U_{\rm uc}$ be an upper cluster algebra with universal coefficient semifield and $t, t_0$  two seeds of $\mathcal U_{\rm uc}$.
Let $\frac{X_{k;t}}{1+X_{k;t}}M_1+\frac{1}{1+X_{k;t}}M_2$ be the $k$-th exchange binomial of $t$ and $\frac{X_{j;t_0}}{1+X_{j;t_0}}N_1+\frac{1}{1+X_{j;t_0}}N_2$  the $j$-th exchange binomial of $t_0$. Then the following conditions are equivalent.

\begin{itemize}
\item[(i)] $X_{k;t}=X_{j;t_0}$;

\item[(ii)] $(A_{k;t}, M_1, M_2)=(A_{j;t_0}, N_1, N_2)$;

\item[(iii)] $(A_{k;t},\{A_{i;t}\mid b_{ik}^t\neq 0\})=(A_{j;t_0},\{A_{i;t_0}\mid b_{ij}^{t_0}\neq 0\})$;

\item[(iv)] $A_{k;t}=A_{j;t_0}$ and $\{A_{i;t}\mid b_{ik}^t\neq 0\}\subseteq A_{t_0}$;

\item[(v)] $(A_{k;t},A_{k;t^\prime})=(A_{j;t_0},A_{j;t_1})$, where $t^\prime=\mu_k(t)$ and $t_1=\mu_j(t_0)$;

\item[(vi)] $A_{k;t}=A_{j;t_0}$ and $\frac{X_{k;t}}{1+X_{k;t}}M_1+\frac{1}{1+X_{k;t}}M_2=\frac{X_{j;t_0}}{1+X_{j;t_0}}N_1+\frac{1}{1+X_{j;t_0}}N_2$.
\end{itemize}
\end{theorem}
\begin{proof}
(i)$\Longrightarrow$ (ii):
By Theorem \ref{thmsf} (ii), we know that $$X_{k;t}=X_{t_0}^{{\bf c}_{k;t}^{B_{t_0};t_0}}\cdot \prod\limits_{i=1}^n(F_{i;t}^{B_{t_0};t_0}(X_{1;t_0},\ldots,X_{n;t_0}))^{b_{ik}^t}.$$
By $X_{k;t}=X_{j;t_0}$ and Proposition \ref{profy}, we get that ${\bf c}_{k;t}^{B_{t_0};t_0}$ is equal to the $j$-th column vector of $I_n$ and
$$\prod\limits_{i=1}^n(F_{i;t}^{B_{t_0};t_0}(X_{1;t_0},\ldots,X_{n;t_0}))^{b_{ik}^t}=1.$$
So we get that
\begin{eqnarray}\label{eqnfequal1}
\prod\limits_{b_{ik}^t>0}(F_{i;t}^{B_{t_0};t_0})^{b_{ik}^t}=
\prod\limits_{b_{ik}^t<0}(F_{i;t}^{B_{t_0};t_0})^{-b_{ik}^t}.
\end{eqnarray}

Let $\overline{\mathcal U}$ be the  upper cluster algebra with principal coefficients at $t_0$ and with initial exchange matrix $B_{t_0}$. We use $(\overline B_t,\overline X_t,\overline A_t)$ to denote the seed of $\overline{\mathcal U}$ at vertex $t$. By $\overline B_{t_0}=B_{t_0}$, we know that $\overline B_t=B_t$ for any vertex $t$. For any cluster monomial $L=A_{1;t}^{v_1}\cdots A_{n;t}^{v_n}$ of $\mathcal U_{\rm uc}$, we denote by $$\overline L=\overline A_{1;t}^{\;v_1}\cdots \overline A_{n;t}^{\;v_n}$$ the corresponding cluster monomial of $\overline{\mathcal U}$ and we denote by $F_{\overline L}^{B_{t_0};t_0}$ the $F$-polynomial of $\overline L$. By the assumption, we know that  $M_1=\prod\limits_{b_{ik}^t>0}A_{i;t}^{b_{ik}^t}$ and $M_2=\prod\limits_{b_{ik}^t<0}A_{i;t}^{-b_{ik}^t}$.
So we have $$\overline M_1=\prod\limits_{b_{ik}^t>0}\overline A_{i;t}^{\;b_{ik}^t}\;\;\text{and}\;\;\;\overline M_2=\prod\limits_{b_{ik}^t<0}\overline A_{i;t}^{\;-b_{ik}^t}.$$

We claim that  $M_1$ and $M_2$ are cluster monomials in $A_{t_0}$. In order to prove this claim, we first show that  $\overline M_1$ and $\overline M_2$ are cluster monomials in $\overline A_{t_0}$.

By Theorem \ref{thmup} (iii), we  can always write the cluster monomial $\overline M_s$ ($s=1,2$) as $$\overline M_s=\overline M_s^{\;\prime\prime}\cdot \overline M_s^{\;\prime},$$
where $\overline M_s^{\;\prime\prime}$ is a cluster monomial in initial cluster variables of $\overline{\mathcal U}$ and $\overline M_s^{\;\prime}$ is a cluster monomial in non-initial cluster variables of $\overline{\mathcal U}$. By the equality (\ref{eqnfequal1}), we know that $F_{\overline M_1}^{B_{t_0};t_0}=F_{\overline M_2}^{B_{t_0};t_0}$. Because $\overline M_1^{\;\prime\prime}$ and $\overline M_2^{\;\prime\prime}$ are cluster monomials in initial cluster variables, we have $F_{\overline M_1^{\;\prime\prime}}^{B_{t_0};t_0}=F_{\overline M_2^{\;\prime\prime}}^{B_{t_0};t_0}=1.$
So  $$F_{\overline M_1^{\;\prime}}^{B_{t_0};t_0}=F_{\overline M_1}^{B_{t_0};t_0}=F_{\overline M_2}^{B_{t_0};t_0}=F_{\overline M_2^{\;\prime}}^{B_{t_0};t_0}.$$
Then by Theorem \ref{thmfpoly} and the fact that $\overline M_1^{\;\prime}$ and  $\overline M_2^{\;\prime}$ are two cluster monomials in non-initial cluster variables, we get $\overline M_1^{\;\prime}=\overline M_2^{\;\prime}$. Because $\overline M_1$ and $\overline M_2$ have no common factor in $\mathbb Z[\overline A_{1;t},\ldots,\overline A_{n;t}]$, the same holds for $\overline M_1^{\;\prime}$ and $\overline M_2^{\;\prime}$. So $\overline M_1^{\;\prime}=\overline M_2^{\;\prime}$ implies that  $$\overline M_1^{\;\prime}=\overline M_2^{\;\prime}=1.$$
Thus $\overline M_1=\overline M_1^{\;\prime\prime}$ and $\overline M_2=\overline M_2^{\;\prime\prime}$ are cluster monomials in $\overline A_{t_0}$. Then by Proposition \ref{proexb} (i), we get that  $M_1$ and $M_2$ are cluster monomials in $A_{t_0}$.

By Proposition \ref{proxp}, we know that $\widehat {\mathcal S}=\{(B_{t},\widehat X_{t})\}_{t\in\mathbb T_n}$ forms an $\mathscr X$-seed pattern, where $\widehat X_t=(\widehat X_{1;t},\ldots,\widehat X_{n;t})$ is given by $\widehat X_{k; t}:=X_{k;t}\prod\limits_{i=1}^nA_{i;t}^{b_{ik}^{t}}$. Since $\mathcal U_{\rm uc}$ is an upper cluster algebra with universal coefficient semifield, $X_{k;t}=X_{j;t_0}$ implies that $\widehat X_{k;t}=\widehat X_{j;t_0}$. So we have $X_{k;t}\prod\limits_{i=1}^nA_{i;t}^{b_{ik}^{t}}=X_{j;t_0}\prod\limits_{i=1}^nA_{i;t_0}^{b_{ij}^{t_0}}$. Thus we get that
$$\frac{M_1}{M_2}=\prod\limits_{i=1}^nA_{i;t}^{b_{ik}^{t}}=\prod\limits_{i=1}^nA_{i;t_0}^{b_{ij}^{t_0}}=\frac{N_1}{N_2}.$$
Since $M_1,M_2,N_1,N_2$ are monomials in $\mathbb Z[A_{1;t_0},\ldots,A_{n;t_0}]$ and $(M_1,M_2)$ and $(N_1,N_2)$ are relatively prime in $\mathbb Z[A_{1;t_0},\ldots,A_{n;t_0}]$, we get $$M_1=N_1\;\;\text{ and }\;\;\;M_2=N_2.$$  Then by $X_{k;t}=X_{j;t_0}$, we have $$\frac{X_{k;t}}{1+X_{k;t}}M_1+\frac{1}{1+X_{k;t}}M_2=\frac{X_{j;t_0}}{1+X_{j;t_0}}N_1+\frac{1}{1+X_{j;t_0}}N_2.$$
By Lemma \ref{lemb}, we get that  $$A_{k;t}=A_{j;t_0}\;\;\text{ or }\;\;\;A_{k;t}=A_{j;t_1},$$
where $t_1=\mu_j(t_0)$.

Now we rule out the possibility of $A_{k;t}=A_{j;t_1}$ with the help of the condition  $X_{k;t}=X_{j;t_0}$.  Assume by contradiction that $A_{k;t}=A_{j;t_1}$. Then the $k$-th column of the $G$-matrix $G_t^{B_{t_0};t_0}$ equals the $j$-th column of $G_{t_1}^{B_{t_0};t_0}$. In particular, we have $g_{jk;t}^{B_{t_0;t_0}}=g_{jj;t_1}^{B_{t_0;t_0}}$. By $t_1=\mu_j(t_0)$, we get that
$$g_{jk;t}^{B_{t_0;t_0}}=g_{jj;t_1}^{B_{t_0;t_0}}=-1.$$
By Theorem \ref{thmcg}, we know that
 \begin{eqnarray}\label{eqncg1}
 S^{-1}(G_t^{B_{t_0};t_0})^{\rm T}S\cdot C_t^{B_{t_0};t_0}=I_n,
 \end{eqnarray}
 where  $S=\mathrm{diag}(s_1,\dots,s_n)$ is a skew-symmetrizer of $B_{t_0}$. By $X_{k;t}=X_{j;t_0}$, we know that the $k$-th column of the $C$-matrix $C_t^{B_{t_0};t_0}$ equals the $j$-th column of $I_n$. By comparing the $k$-th column of the two sides of the equality (\ref{eqncg1}), we get that the $j$-th column of $ S^{-1}(G_t^{B_{t_0};t_0})^{\rm T}S$ equals the $k$-th column of $I_n$. In particular, the $(k,j)$-entry $s_k^{-1}\cdot g_{jk;t}^{B_{t_0;t_0}}\cdot s_j$ of $S^{-1}(G_t^{B_{t_0};t_0})^{\rm T}S$ equals $1$, and thus $$g_{jk;t}^{B_{t_0;t_0}}=\frac{s_k}{s_j}>0.$$
This contradicts $g_{jk;t}^{B_{t_0;t_0}}=-1$.  So $A_{k;t}\neq A_{j;t_1}$ and thus $A_{k;t}=A_{j;t_0}$.

Hence, $X_{k;t}=X_{j;t_0}$ implies $(A_{k;t}, M_1, M_2)=(A_{j;t_0}, N_1, N_2)$.

 (ii)$\Longrightarrow$ (iii): It suffices to show that $$\{A_{i;t}\mid b_{ik}^t\neq 0\}=\{A_{i;t_0}\mid b_{ij}^{t_0}\neq 0\}.$$
 Let $\mathbb A$ be the set of cluster variables of $\mathcal U_{\rm uc}$ and $P\in \mathbb A$. We first claim that $(P\mid\mid  M_1\cdot M_2)_d<0$ if and only if $P$ is a cluster variable appearing in the cluster monomial $M_1\cdot M_2$.

 If $P$ is a cluster variable appearing in the cluster monomial $M_1\cdot M_2$, then it is easy to see that  $(P\mid\mid  M_1\cdot M_2)_d<0$. Conversely, if $(P\mid\mid  M_1\cdot M_2)_d<0$, then by Corollary \ref{cordmonomial} (i), we know that $P$ is a cluster variable appearing in the cluster monomial $M_1\cdot M_2$.

 By the discussion above, we get
 $$\{A_{i;t}\mid b_{ik}^t\neq 0\}=\{P\in \mathbb A\mid (P\mid\mid  M_1\cdot M_2)_d<0\}.$$
Similarly, we have  $$\{A_{i;t_0}\mid b_{ij}^{t_0}\neq 0\}=\{P\in \mathbb A\mid (P\mid\mid  N_1\cdot N_2)_d<0\}.$$
By $(A_{k;t}, M_1, M_2)=(A_{j;t_0}, N_1, N_2)$, we have $M_1\cdot M_2=N_1\cdot N_2$. Thus we get $$\{A_{i;t}\mid b_{ik}^t\neq 0\}=\{A_{i;t_0}\mid b_{ij}^{t_0}\neq 0\}.$$

(iii)$\Longrightarrow$ (iv): This is clear.

 (iv)$\Longrightarrow$ (v):
We know that
 \begin{eqnarray}
 A_{k;t^\prime}=A_{k;t}^{-1}(\frac{X_{k;t}}{1+X_{k;t}}M_1+\frac{1}{1+X_{k;t}}M_2).\nonumber
 \end{eqnarray}
By $A_{k;t}=A_{j;t_0}$ and the fact that $M_1, M_2$ are monomials in variables from $$\{A_{i;t}\mid b_{ik}^t\neq 0\}\subseteq A_{t_0},$$ we know that $$ A_{k;t^\prime}=A_{k;t}^{-1}(\frac{X_{k;t}}{1+X_{k;t}}M_1+\frac{1}{1+X_{k;t}}M_2)
=A_{j;t_0}^{-1}(\frac{X_{k;t}}{1+X_{k;t}}M_1+\frac{1}{1+X_{k;t}}M_2)$$
 can be viewed as the
 Laurent expansion of $A_{k;t^\prime}$ with respect to $A_{t_0}$.
Thus we know that $$(A_{j;t_0}\mid\mid  A_{k;t^\prime})_d=1>0\;\text{ and }(A_{i;t_0}\mid\mid  A_{k;t^\prime})_d=0,$$
for any $i\neq j$. Then by Proposition \ref{lemdex}, we get $A_{k;t^\prime}=A_{j;t_1}$. So $$(A_{k;t},A_{k;t^\prime})=(A_{j;t_0},A_{j;t_1}).$$

(v)$\Longrightarrow$ (vi): This is clear.

 (vi)$\Longrightarrow$ (i): If the $j$-th column  of $B_{t_0}$ is zero, then the same holds for any exchange matrix of $\mathcal U_{\rm uc}$. In this case,
 the action of $\mu_j$ on any seed is very clear. We can see that
 $A_{k;t}=A_{j;t_0}$ implies that $k=j$ and $X_{k;t}=X_{j;t_0}$.

 If the $j$-th column of $B_{t_0}$ is non-zero, then $N_1$ and $N_2$ are two different cluster monomials in $A_{t_0}$. Then by Proposition \ref{protwo} and the following equality,
  $$\frac{X_{j;t_0}}{1+X_{j;t_0}}N_1+\frac{1}{1+X_{j;t_0}}N_2=\frac{X_{k;t}}{1+X_{k;t}}M_1+\frac{1}{1+X_{k;t}}M_2,$$
  we  get $\frac{X_{k;t}}{1+X_{k;t}}=\frac{X_{j;t_0}}{1+X_{j;t_0}}$ or $\frac{X_{k;t}}{1+X_{k;t}}=\frac{1}{1+X_{j;t_0}}$. Namely, we get that $$X_{k;t}=X_{j;t_0}\;\;\text{ or }\;\;\;X_{k;t}=X_{j;t_0}^{-1}.$$

Similar to the argument for the statement (i)$\Longrightarrow$ (ii), now we rule out the possibility of $X_{k;t}=X_{j;t_0}^{-1}$ with the help of the condition $A_{k;t}=A_{j;t_0}$. Assume by contradiction that $X_{k;t}=X_{j;t_0}^{-1}$. Then the $k$-th column of $C_t^{B_{t_0};t_0}$ is the $j$-th column of $-I_n$, by Theorem \ref{thmsf} (ii). So we have $$c_{jk;t}^{B_{t_0};t_0}=-1<0.$$

 By Theorem \ref{thmcg}, we know that
 \begin{eqnarray}\label{eqncg}
 (G_t^{B_{t_0};t_0})^{\rm T}(SC_t^{B_{t_0};t_0}S^{-1})=I_n,
 \end{eqnarray}
 where  $S=\mathrm{diag}(s_1,\dots,s_n)$ is a skew-symmetrizer of $B_{t_0}$.

 By $A_{k;t}=A_{j;t_0}$, we know that the $k$-th column  of $G_t^{B_{t_0};t_0}$ equals the $j$-th column vector of $I_n$. So the $k$-th row of $(G_t^{B_{t_0};t_0})^{\rm T}$ equals the $j$-th row of $I_n$. By comparing the $k$-th row of the two sides of the equality (\ref{eqncg}), we know that the $j$-th row of $SC_t^{B_{t_0};t_0}S^{-1}$ equals the $k$-th row of $I_n$. In particular, the $(j,k)$-entry $s_jc_{jk;t}^{B_{t_0};t_0}s_{k}^{-1}$ of $SC_t^{B_{t_0};t_0}S^{-1}$ equals $1$ and thus $$c_{jk;t}^{B_{t_0};t_0}=\frac{s_{k}}{s_j}>0.$$
 This contradicts $c_{jk;t}^{B_{t_0};t_0}=-1<0$. So $X_{k;t}\neq X_{j;t_0}^{-1}$ and thus  $X_{k;t}=X_{j;t_0}$.
\end{proof}

Let $\mathcal U$ be an upper cluster algebra and $\mathcal X$ a cluster Poisson algebra. We say that $\mathcal U$ and $\mathcal X$ have the {\em same type} if they have the same exchange matrix at a vertex $t_0$ of $\mathbb T_n$.

\begin{theorem}\label{thmbij}
Let $\mathcal X$ be a cluster Poisson algebra and $\mathcal U$ an upper cluster algebra of the same type as $\mathcal X$. Denote $\overline X_t$  the Poisson cluster of $\mathcal X$ at $t$ to distinguish the $\mathscr X$-cluster $X_t$ of $\mathcal U$ at $t$. Each mutation $t_k=\mu_k(t)$ gives an $\mathscr X$-exchange pair $(\overline X_{k;t},\overline X_{k;t_k})$ of $\mathcal X$ and an $\mathscr A$-exchange pair $(A_{k;t},A_{k;t_k})$ of $\mathcal U$. Then

\begin{itemize}
\item[(i)] The map $\psi:\overline X_{k;t}\mapsto (A_{k;t},A_{k;t_k})$ gives a bijection from the set of cluster Poisson variables of $\mathcal X$ to the set of  $\mathscr A$-exchange pairs of $\mathcal U$;

\item[(ii)] The map $\psi_1:(\overline X_{k;t}, \overline X_{k;t_k})\mapsto (A_{k;t},A_{k;t_k})$ gives a bijection from the set of $\mathscr X$-exchange pairs of $\mathcal X$ to  the set of $\mathscr A$-exchange pairs of $\mathcal U$.
\end{itemize}
\end{theorem}

\begin{proof}
(i) Thanks to Proposition \ref{proexb}, we can just assume that $\mathcal U=\mathcal U_{\rm uc}$ is an upper cluster algebra with universal coefficient semifield. In this case, the cluster Poisson variables of $\mathcal X$ can be viewed as the $\mathscr X$-variables of $\mathcal U=\mathcal U_{\rm uc}$. Then (i) follows from Theorem \ref{mainthm} (i)(v).

(ii) This follows from (i) and the fact that in any $\mathscr X$-exchange pair $(\overline X_{k;t},\overline X_{k;t_k})$, one has $\overline X_{k;t_k}=\overline X_{k;t}^{-1}$.
\end{proof}
\begin{remark}
For (upper) cluster algebras of finite type, a different proof of Theorem \ref{thmbij} can be found in \cite[Corollary 1.2]{S}.
\end{remark}

\subsection{Exchange graphs of cluster Poisson algebras}\label{sec72}

Let $\mathcal X$ be a cluster Poisson algebra. In this subsection, we prove that  the  seeds of  $\mathcal X$ whose Poisson clusters contain particular cluster Poisson variables form a connected subgraph of the exchange graph of $\mathcal X$.

\begin{lemma}\label{lemxi}
Let $\mathcal U_{\rm uc}$ be an upper cluster algebra with universal coefficient semifield, and let $t_0$ and $t_1$ be two seeds of $\mathcal U_{\rm uc}$ with $t_1=\mu_k(t_0)$. For a given $1\leq j\leq n$, denote $I(t_0)=\{j\}\sqcup\{i\mid b_{ij}^{t_0}\neq 0\}$ and $I(t_1)=\{j\}\sqcup\{i\mid b_{ij}^{t_1}\neq 0\}$. If
$k\notin I(t_0)$, then we have $X_{j;t_0}=X_{j;t_1}$ and $I(t_0)=I(t_1)$.
\end{lemma}
\begin{proof}
The results follow from $k\notin I(t_0)$ and the definition of mutation.
\end{proof}

\begin{lemma}\label{lemiff}
Let $\mathcal U_{\rm uc}$ be an upper cluster algebra with universal coefficient semifield and $t, t_0$  two seeds of $\mathcal U_{\rm uc}$. Then $X_{j;t_0}\in X_{t}$ if and only if
$\{A_{j;t_0}\}\sqcup\{A_{i;t_0}\mid  b_{ij}^{t_0}\neq 0\}\subseteq A_{t}$.
\end{lemma}
\begin{proof}
$\Longrightarrow$: This follows from Theorem \ref{mainthm} (i)(iii).

$\Longleftarrow$: By $\{A_{j;t_0}\}\sqcup\{A_{i;t_0}\mid  b_{ij}^{t_0}\neq 0\}\subseteq A_{t}$, we know that $$U:=\{A_{j;t_0}\}\sqcup\{A_{i;t_0}\mid  b_{ij}^{t_0}\neq 0\}$$ is a common subset of $A_t$ and $A_{t_0}$. Then by Theorem \ref{thmacon}, there exists a  sequence of mutations $(\mu_{k_1},\ldots,\mu_{k_s})$ such that the $\mathscr A$-seed at $t_s=\mu_{k_s}\cdots\mu_{k_1}(t_0)$ is equivalent to the $\mathscr A$-seed at $t$ and the cluster variables in $U$ remain unchanged when we do each mutation along the sequence of mutations $(\mu_{k_1},\ldots,\mu_{k_s})$. Set $I=\{j\}\sqcup\{i\mid b_{ij}^{t_0}\neq 0\}$. We know that $k_{\ell}\notin I$ for any $\ell=1,\ldots,s$. Applying Lemma \ref{lemxi} iteratively, we get $X_{j;t_0}\in X_{t_s}$. Because the two seeds $t$ and $t_s$ are equivalent,  we get $X_{j;t_0}\in X_t$.
\end{proof}

\begin{theorem}\label{thmgraph}
Let $\mathcal X$ be a cluster Poisson algebra. Then the  seeds of  $\mathcal X$ whose Poisson clusters contain particular cluster Poisson variables form a connected subgraph of the exchange graph of $\mathcal X$.
\end{theorem}

\begin{proof}
Let $\mathcal U=\mathcal U_{\rm uc}$ be the upper cluster algebra with universal coefficient semifield of the same type as $\mathcal X$. In this case, the cluster Poisson variables (respectively, Poisson clusters) of $\mathcal X$ can be viewed as the $\mathscr X$-variables (respectively, $\mathscr X$-clusters) of $\mathcal U=\mathcal U_{\rm uc}$. Now we work on the upper cluster algebra $\mathcal U$.

 Fix a subset $V$ of $X_{t_0}$.  Let $t$ be any seed of $\mathcal U$ such that its $\mathscr X$-cluster $X_t$ satisfies $V\subseteq X_t$. It suffices for us to show that there exists a sequence of mutations $(\mu_{k_1},\ldots,\mu_{k_s})$ such that the $\mathscr X$-seed at $t_s=\mu_{k_s}\cdots\mu_{k_1}(t_0)$ is equivalent to the $\mathscr X$-seed at $t$ and the $\mathscr X$-variables in $V$ remain unchanged when we do each mutation along the sequence of mutations $(\mu_{k_1},\ldots,\mu_{k_s})$.

Let $U:=\bigcup\limits_{X_{j;t_0}\in V}(\{A_{j;t_0}\}\sqcup\{A_{i;t_0}\mid b_{ij}^{t_0}\neq 0\})$. By Lemma \ref{lemiff}, we know that for a seed $t^\prime$ of $\mathcal U$, its $\mathscr X$-cluster $X_{t^\prime}$ satisfies $V\subseteq X_{t^\prime}$ if and only if its cluster $A_{t^\prime}$ satisfies $U\subseteq A_{t^\prime}$.

By $X_{t_0}\supseteq V$ and $X_t\supseteq V$, we get that $A_{t_0}\supseteq U$ and $A_t\supseteq U$.
 Then by Theorem \ref{thmacon}, there exists a  sequence of mutations $(\mu_{k_1},\ldots,\mu_{k_s})$ such that the $\mathscr A$-seed at $t_s=\mu_{k_s}\cdots\mu_{k_1}(t_0)$ is equivalent to the $\mathscr A$-seed at $t$ and the cluster variables in $U$ remain unchanged when we do each mutation along the sequence of mutations $(\mu_{k_1},\ldots,\mu_{k_s})$. So the $\mathscr X$-seed at $t_s=\mu_{k_s}\cdots\mu_{k_1}(t_0)$ is equivalent to the $\mathscr X$-seed at $t$ and the $\mathscr X$-variables in $V$ remain unchanged when we do each mutation along the sequence of mutations $(\mu_{k_1},\ldots,\mu_{k_s})$. This completes the proof.
\end{proof}

\section*{Appendix: The proof of Proposition \ref{pro:atomic}}\label{app}

In this appendix, we give the proof of Proposition \ref{pro:atomic}. The main idea for the proof follows that of \cite{GS_2022} with some modifications and simplifications.

Keep the notations in Section \ref{sec:40}. In particular, we have   ${\bf inv}=\{Z_{p+1},\cdots,Z_m\}$ and
$${\mathcal L}(t,{\bf inv})=\mathbb K[Z_1,\cdots,Z_p,Z_{p+1}^{\pm1},\cdots,Z_m^{\pm1}][A_{1;t}^{\pm1},\cdots,A_{n;t}^{\pm1}].$$
Clearly, ${\mathcal L}(t,{\bf inv})$ is factorial. Thus any  $0\neq M\in {\mathcal L}(t,{\bf inv})$ has a decomposition $M=r\cdot a_1a_2\cdots a_s$ satisfying that
\begin{itemize}
    \item $r$ is invertible in ${\mathcal L}(t,{\bf inv})$;
    \item each $a_i$ is irreducible in ${\mathcal L}(t,{\bf inv})$.
\end{itemize}
Moreover, the non-negative integer $s$ only depends on $M$, not on the choice of such a decomposition. Hence, $$\lambda_t:M\mapsto s$$ defines a map from ${\mathcal L}(t,{\bf inv})\setminus\{0\}\rightarrow \mathbb N$. The following result is clear from the definition of $\lambda_t$.

\begin{lemma}\label{lem:app1}
\begin{itemize}
    \item [(i)] $\lambda_t(M)=0$ if and only if $M$ is invertible in  ${\mathcal L}(t,{\bf inv})$.
    \item[(ii)] $\lambda_t(M\cdot N)=\lambda_t(M)+\lambda_t(N)$ for any  $M,N\in {\mathcal L}(t,{\bf inv})\setminus\{0\}$.
\end{itemize}
\end{lemma}

Fix a seed $t_0$ of $\mathcal U$. Let $t_k=\mu_k(t_0)$ for $k=1,\ldots,n$ and denote by
$$\widetilde {\mathcal U}({\bf inv})= \bigcap_{j=0}^n\mathcal L(t_j,{\bf inv}).$$

\begin{lemma} \label{lem:app2}
\begin{itemize}
    \item [(i)] $\mathcal A({\bf inv})\subseteq \mathcal U({\bf inv})\subseteq \widetilde {\mathcal U}({\bf inv})\subseteq \mathcal L(t_0,{\bf inv})$.
    \item[(ii)] Let $R$ be one of the three domains
$\mathcal A({\bf inv}),\; \mathcal U({\bf inv}),\;\widetilde {\mathcal U}({\bf inv})$. Then the set $R^\times$ of invertible elements in $R$ is 
    $$R^\times =\{\kappa Z_{p+1}^{c_{p+1}}\cdots Z_m^{c_m}\mid \kappa\in\mathbb K^\times,\; c_{p+1},\ldots,c_m\in \mathbb Z\}.$$
\end{itemize}
\end{lemma}
\begin{proof}
    (i) This is clear. 
    
    (ii) The proof is the same as that of \cite[Theorem 2.2]{GLS13}.
\end{proof}

Recall that for each seed $t$ of $\mathcal U({\bf inv})$, we defined a map $\lambda_t$. In particular, 
 for the seed $t_j$, we have the map $\lambda_{t_j}:\mathcal L(t_j,{\bf inv})\setminus\{0\}\rightarrow \mathbb N$, where $j=0,\ldots,n$. Let 
 \begin{eqnarray}\label{eqn:lambda}
    \lambda:\widetilde {\mathcal U}({\bf inv})\setminus\{0\} \rightarrow \mathbb N \nonumber
 \end{eqnarray}
 be the map defined by $\lambda(M):=\sum_{j=0}^n\lambda_{t_j}(M)$, where $M\in \widetilde {\mathcal U}({\bf inv})\setminus\{0\}$.


 \begin{lemma}\label{lem:app3}
     \begin{itemize}
         \item [(i)] $\lambda(M)=0$ if and only if $M$ is invertible in  $\widetilde {\mathcal U}({\bf inv})$.
         \item [(ii)]  $\lambda(M\cdot N)=\lambda(M)+\lambda(N)$ for any  $M,N\in\widetilde {\mathcal U}({\bf inv})\setminus\{0\}$.
     \end{itemize}
 \end{lemma}
\begin{proof}
(i) Since $\lambda(M)=\sum_{j=0}^n \lambda_{t_j}(M)$ and each $\lambda_{t_j}(M)$ takes non-negative values, we know that $\lambda(M)=0$ if and only if $\lambda_{t_j}(M)=0$ for any $j$. By Lemma \ref{lem:app1} (i), we know that $\lambda_{t_j}(M)=0$ for any $j$ if and only if $M$ is invertible in $\mathcal L(t_j,{\bf inv})$ for any $j$. This happens if and only if $M$ is invertible in $\bigcap_{j=0}^n\mathcal L(t_j,{\bf inv})=\widetilde {\mathcal U}({\bf inv})$.

(ii) This follows from Lemma \ref{lem:app1} (ii).
\end{proof}

\begin{lemma}\label{lem:app4}
    Let $R$ be a non-zero subring of $\widetilde {\mathcal U}({\bf inv})$. Suppose whenever an element $r\in R$ is invertible in $\widetilde {\mathcal U}({\bf inv})$, it is invertible in $R$. Then $R$ is atomic. In particular, $\widetilde {\mathcal U}({\bf inv})$ itself is atomic.
\end{lemma}

Indeed, the function $\lambda$ is a very special case of \emph{length function} on $R$ in the sense of \cite[Definition 1.1.3]{geroldinger2006non}, and the atomicity of $R$ is implied by \cite[Proposition 1.1.4]{geroldinger2006non}. Let us give a proof for the convenience of the reader.

\begin{proof}

Fix a non-zero, non-invertible element $r\in R$. We need to show that it can be decomposed as a product of irreducible elements in $R$.

Let $r=a_1\cdots a_s$ be a decomposition of $r$ such that each $a_i$ is a non-invertible element in $R\setminus\{0\}$. Such a decomposition exists, because we can always take $s=1$ and $a_1=r$.

Fix a decomposition $r=a_1\cdots a_s$ as above. This decomposition can be viewed as decomposition in $\widetilde {\mathcal U}({\bf inv})\setminus\{0\}$. Then by Lemma \ref{lem:app3} (ii), we have $$\lambda(r)=\lambda(a_1)+\ldots +\lambda(a_s).$$
Since each $a_i$ is not invertible in $R$ and, by the assumption, we know that it is also not invertible in $\widetilde {\mathcal U}({\bf inv})$. Then by Lemma \ref{lem:app3} (i), we have $\lambda(a_i)\in\mathbb Z_{\geq 1}$.
Hence, 
\begin{eqnarray}\label{eqn:r>s}
    \lambda(r)=\lambda(a_1)+\ldots +\lambda(a_s)\geq s.\nonumber
\end{eqnarray}

Since $s$ is bounded by $\lambda(r)$, there exists a decomposition $r=a_1\cdots a_s$ such that $s$ is maximal. We claim that in such a decomposition, each $a_i$ is irreducible in $R$. Otherwise, 
there exists some $i_0$ such that $a_{i_0}$ is not irreducible in $R$. Without loss of generality, we can assume $i_0=s$.
Then there exist two non-invertible elements $a_s',a_s''$ in $R\setminus\{0\}$ such that $a_s=a_s'a_s''$. Thus we have $$r=a_1\cdots a_{s-1}a_s'a_s'',$$ which is a decomposition of $r$ of length $s+1$ satisfying that each factor is not invertible in $R\setminus\{0\}$. This contradicts the choice of $s$. Hence,  $r=a_1\cdots a_s$ is a decomposition of $r\in R$ satisfying that each $a_i$ is irreducible in $R$. Therefore, $R$ is atomic.
\end{proof}

\begin{proof}[Proof of Proposition \ref{pro:atomic}] Thanks to Lemma \ref{lem:app2}, we can apply  Lemma \ref{lem:app4} to the case $R=\mathcal U({\bf inv})$ or $R=\mathcal A({\bf inv})$. Hence, $\mathcal U({\bf inv})$ and $\mathcal A({\bf inv})$ are atomic.
\end{proof}

\bibliographystyle{alpha}
\bibliography{myref}
\end{document}